\documentclass[pdflatex]{sn-jnl}

\usepackage{graphicx}%
\usepackage{multirow}%
\usepackage{amsmath,amssymb,amsfonts}%
\usepackage{amsthm}%
\usepackage{mathrsfs}%
\usepackage[title]{appendix}%
\usepackage{xcolor}%
\usepackage{textcomp}%
\usepackage{manyfoot}%
\usepackage{booktabs}%
\usepackage{algorithm}%
\usepackage{algorithmicx}%
\usepackage{algpseudocode}%
\usepackage{listings}%
\usepackage{natbib}

\renewcommand{\proofname}{Proof.}

\newtheorem{theorem}{Theorem}[section]
\newtheorem{corollary}[theorem]{Corollary}
\newtheorem{lemma}[theorem]{Lemma}
\newtheorem{remark}[theorem]{Remark}

\newtheorem{proposition}[theorem]{Proposition}%
\newtheorem{assumption}[theorem]{Assumption}%
\newtheorem{definition}[theorem]{Definition}

\begin{document}

\title[A Variational Framework for the Complexity of PDE Solutions]{A Variational Framework for the Complexity of PDE Solutions}


\author*[1,2]{\fnm{Juan Esteban} \sur{Suarez Cardona}}\email{suarez@math.lmu.de}

\author[4,5,6]{\fnm{Holger} \sur{Boche}}\email{boche@tum.de}

\author[1,2,3,5]{\fnm{Gitta} \sur{Kutyniok}}\email{kutyniok@math.lmu.de}

\affil*[1]{\orgdiv{Department of Mathematics}, \orgname{Ludwig-Maximilians-Universit\"at M\"unchen, } \country{Germany}}
\affil*[2]{\orgname{Munich Center for Machine Learning (MCML), }\country{Germany}}
\affil[3]{\orgdiv{Department of Physics and Technology},\orgname{University of Tromsø, }\country{Norway}}
\affil[4]{\orgdiv{Institute of Theoretical Information Technology}, \orgname{TUM School of Computation, Information and Technology, Technical University of Munich}, \country{Germany}}
\affil[5]{\orgname{Munich Center for Quantum Science and Technology (MCQST)}, \country{Germany}}
\affil[6]{\orgname{Munich Quantum Valley (MQV)}, \country{Germany}}



\abstract{Partial Differential Equations (PDEs) are fundamental mathematical models for describing physical phenomena, yet most PDEs of practical interest cannot be solved analytically and require numerical approximations. The feasibility of such numerical methods, however, is ultimately constrained by the limitations of existing computational models. Since digital computers constitute the primary physical realizations of numerical computations, and Turing machines define their theoretical limits, the question of Turing computability of PDE solutions arises as a problem of fundamental theoretical significance. The Turing computability of PDE solutions provides a rigorous framework to distinguish equations that are, in principle, effectively solvable from those that inherently encode undecidable or non-computable behavior.
Once computability has been established, complexity theory extends the analysis by quantifying the computational resources required to approximate the corresponding PDE solutions. In this work, we present a novel framework based on least-squares variational formulations and their associated gradient flows to analyze the computability and complexity of PDE solutions from an optimization perspective. Our approach enables the approximation of PDE solution operators via discrete gradient flows, linking structural properties of the PDE, such as coercivity, ellipticity, and convexity, to the complexity of their solutions. 

Within this setting, we characterize representation- and discretization-dependent sufficient conditions for regimes where PDEs admit effective, polynomial-time numerical approximations, as well as those exhibiting complexity blowup, where the input data possess polynomial-time complexity, yet the solution itself scales super-polynomially.

In summary, this paper develops a general variational framework for analyzing the computability and computational complexity of classes of PDE solutions. The results provide insights into how structural properties of the underlying PDE and the regularity of the solutions influence its complexity, by establishing sufficient conditions for computability and complexity bounds. Beyond the theoretical characterization, the framework provides principled guidelines for the design of effective numerical methods and contributes to the understanding of the fundamental limitations of digital computation for PDE problems.
}




\maketitle

\section{Introduction}\label{sec1}
For centuries, the modeling and description of physical phenomena have relied on the formulation and analysis of Partial Differential Equations (PDEs). While classical analytical methods have provided closed-form solutions for a limited number of cases, the vast majority of physically relevant PDEs require numerical methods to approximate their solutions and simulate the underlying systems. Entire areas within scientific computing and numerical analysis focus on designing and analyzing these numerical methods, ensuring the well-posedness of the problem, and the accuracy and efficiency of the resulting numerical scheme. Yet, the fundamental question of Turing computability \cite{Church1937-sr} of PDE solutions has only been addressed to a limited extent. In this context, the computability question is, however, relevant because digital computers are the main hardware platforms for numerical calculations, with Turing machines defining their theoretical limits. Thus, tasks that are intractable for Turing machines are necessarily more challenging for digital computers in practice. The central problem is to determine whether the infinite-dimensional solutions of PDEs can, in principle, be approximated effectively on a digital machine and, if so, what computational resources are inherently required. Within the framework of computability analysis and complexity theory, we can distinguish PDEs that are effectively solvable from those that may encode undecidable behavior and quantify the efficiency of approximation when computability is assured. Therefore, addressing the computability and complexity of PDE solutions not only clarifies the limits of current digital hardware for solving PDEs but also highlights classes of equations where novel computational architectures \cite{Handler1967-ty} and new notions of computability may be necessary to ensure feasible numerical methods for solving such PDEs.

Classical analyzes of computability and complexity of PDEs have primarily addressed equations admitting closed-form solutions \cite{Pour-El1997-br, Weihrauch2006-cp, bacho2023complexityblowupsolutionslaplace}. However, for most PDEs arising in real-world applications, such explicit solutions are unavailable \cite{Evans2022-ig, Trefethen2000-su}, motivating the development of new frameworks to study their computability and complexity. Recent advances in variational methods, especially least-squares formulations \cite{Los2023-jo, Cai2020-js}, offer a flexible framework that has already shown success in nonlinear, high-dimensional, and even discontinuous problems \cite{Hu2024-sy, Cardona2023-ma, Suarez2025}.

In this work, we leverage the gradient flow associated with least-squares variational formulations \cite{Ghoussoub2004-ou, Ambrosio2008-qd} to develop a novel framework for analyzing the computability and complexity of PDE solutions. Our approach establishes a direct link between the complexity of solving a PDE and the convergence rates of its discrete gradient flow. This connection relates structural properties of the PDE operator, such as coercivity, ellipticity, and convexity, to the resulting complexity class of its solutions, thereby extending existing results to a broader class of PDEs.

The computability and complexity statements obtained in this framework are representation and discretization-dependent, relying on an effective representation of functions via convergent sequences of polynomials generated by a particular discrete gradient flow. Within this setting, complexity is characterized through the rate of convergence of such effective polynomial approximations and the associated computational cost of constructing them. Moreover, the framework establishes a relation between the loss of regularity of PDE solutions and complexity blowup, where polynomial-time input data produce solutions with super-polynomial complexity \cite{bacho2023complexityblowupsolutionslaplace}. A detailed discussion of the scope of the results and their implications for computability and complexity within our variational framework is provided in Subsection~\ref{subsec:Scope}.

To sketch the main idea of the method presented in this paper, we consider a variational loss, over the domain $\Omega\subseteq \mathbb{R}^d$, of the form
\begin{equation}\label{eq:PDE_loss_int}
\mathcal{L}[u] = \| \mathcal{N}[u] - f\|_{H^a(\Omega)}^2 + \|u|_{\partial\Omega} - g\|_{H^s(\partial\Omega)}^2,
\end{equation}
with $\mathcal{N}:H^k(\Omega)\rightarrow H^a(\Omega)$, $f\in H^a(\Omega)$, and $g\in H^s(\partial\Omega)$. The minimizer of $\eqref{eq:PDE_loss_int}$ solves 
\begin{equation}\label{eq:gen_pde_int}
\left\{\begin{array}{cc}
       \mathcal{N}[u] = f,&\textrm{ in }\Omega, \\
      u= g,&\textrm{ on } \partial\Omega,
\end{array}\right.
\end{equation}
in either a strong or weak sense. We construct a numerical framework that solves the discrete gradient flow of the loss $\mathcal{L}$ and produces an optimal surrogate model $\hat{u}_{\theta^*}\in \mathcal{S}_m(\Omega)$, parameterized by the coefficients $\theta^*\in \Theta_m\cong \mathbb{R}^{d_m}$, with $d_m:=\textrm{dim}(\mathcal{S}_m)$, approximating the minimizer $u^*$ of the loss $\mathcal{L}$ in \eqref{eq:PDE_loss_int}.

We further estimate the convergence rate of the $j$-th iterate $\hat{u}_{\theta^j}$, given by a discrete gradient flow of the form
\begin{equation}
    \theta^{j+1} = \theta^j-\delta\tau \nabla \mathcal{L}[\hat{u}_{\theta^j}],
\end{equation}
against the minimizer $u^*$, and provide bounds on the convergence error
\begin{equation}
\mathcal{R}_{l}(j, \mathrm{dim}(\mathcal{S}_m)) \leq  \|\hat{u}_{\theta^j} - u^*\|_{H^a(\Omega)}^2 \leq\mathcal{R}_u(j, \mathrm{dim}(\mathcal{S}_m)),
\end{equation}
where the functions $\mathcal{R}_u$ and $\mathcal{R}_l$ represent the upper and lower bounds on the convergence rates. 
These bounds allow us to estimate the number of computations required for effective convergence, i.e., $\|\hat{u}_{\theta^j} - u^*\|_{H^a(\Omega)}^2 \leq 2^{-N}$, as a function of the precision $N\in \mathbb{N}$, providing an explicit measure of the computational complexity of the PDE solution. Moreover, the lower bound $\mathcal{R}_l$ on the convergence rate allows us to identify PDEs exhibiting complexity blowup, as formalized in Definition~\ref{def:complexity_blowup} and Remark~\ref{rk:complexity_blowup}.

This framework contributes to the development of a computability and complexity theory for general classes of PDEs, extending to both linear and nonlinear equations without closed-form solutions. Its relevance is illustrated through two case studies: the polynomial-time computability of the Poisson equation \cite{Evans2022-ig} and the intrinsic complexity blowup of the Eikonal equation \cite{Alimov2019-np}.

\section{Related Work}
The literature on complexity theory for PDEs focuses mainly on four complexity classes \cite{bacho2023complexityblowupsolutionslaplace}. The class $P$ consists of all decision problems that can be solved in polynomial time by a deterministic Turing machine. The class $NP$ contains those problems for which a proposed solution can be verified in polynomial time. The class $\#P$ describes the counting analogue of $NP$. More precisely, it consists of function problems that, given an input string $x$, return the number of certificates $y$ of length at most polynomial in $|x|$ that can be verified by a polynomial-time Turing machine \cite{Arora2012-bh}. Finally, the class $FP$ includes all the counting problems that can be solved by a function-oracle Turing machine in polynomial-time.

A central question in the complexity theory of PDEs is determining whether solutions lie in a different complexity function class than that of the input data (e.g., boundary conditions or source terms). Under the conjecture that $FP\neq\#P$, the notion of complexity blowup formalizes settings in which the solution of a PDE belongs to a strictly bigger complexity class than its input data. We now review foundational results on complexity blowup in PDEs.

The arguably first study on the computability of PDE solutions, presented in Chapter 3 of~\cite{Pour-El1989-yg}, examined a wave equation of the form
\begin{equation}\label{eq:wave_eq}
\left\{
\begin{aligned}
&\partial_tu +\Delta u = 0,\textrm{ in }\Omega\times(0,T),\\&
u|_{t=0} = u_0,\textrm{ in }\Omega,\\&
\partial_tu\bigr|_{t=0} = 0,\textrm{ in }\Omega,
\end{aligned}\right.
\end{equation}
and demonstrated that for a certain domain $\Omega\times (0,T)\subseteq \mathbb{R}^3\times(0,\infty)$, there exists a computable continuous function $u_0$ such that the solution $u$ is continuous but non-computable. This result was extended in~\cite{Pour-El1997-br}, where it was established that there exists a computable $C^1$ function $u_0$ such that the solution $u$ of \eqref{eq:wave_eq} is nowhere computable on any given compact subset $D\subseteq \mathbb{R}^3\times(0,\infty)$. This result was further generalized in~\cite{Boche2020-uu}, where it was shown that $u_0$ can be chosen as a computable $C^1$ function with compact support in an annulus $\{x\in \mathbb{R}^3: r\leq||x||_2\leq R\}$ for some $r,R\in \mathbb{R}_{+}$, such that its first derivative is absolutely continuous.

Additional work in \cite{Weihrauch2001-rm} addressed the inhomogeneous linear Schr\"odinger equation
\begin{equation}
    \left\{\begin{aligned}
        &\partial_tu = i\Delta u +\phi\mathrm{, in }\,\, \mathbb{R}^d\times (0,\infty),\\&
        u|_{t=0} = u_0 \mathrm{, in }\,\, \mathbb{R}^d,
    \end{aligned}\right.
\end{equation}
showing that the solution $u \in H^s(\mathbb{R}^d\times (0,\infty))$ is computable for all $\phi\in C^0(\mathbb{R}^d\times (0,\infty))$. Moreover, in the homogeneous case ($\phi = 0$), it was shown that the solution $u \in L^p(\mathbb{R}^d\times (0,\infty))$ is computable if and only if $p = 2$. Extensions to the nonlinear setting are presented in \cite{Weihrauch2006-cp}, establishing the computability of solutions in $H^1(\mathbb{R}^d\times (0,\infty))$.

Another prominent PDE that has been studied in the context of computability and complexity analysis is the Poisson equation with Dirichlet boundary conditions 
\begin{equation}
    \left\{\begin{aligned}
        &\Delta u = f,\textrm{ in }\,\,    
        \Omega\subseteq\mathbb{R}^d,\\&
        u = g\textrm{, on}\,\,\partial\Omega.
    \end{aligned}\right.
\end{equation}
It was shown in \cite{Kawamura2017-dv} that if the complexity classes $FP$ and $\#P$ coincide (a stronger assumption than $P = NP$), then any solution $u \in C^2(B_1)$ over the unit ball $B_1 := \{ x \in \mathbb{R}^d : |x|_2 \leq 1\}$ is polynomial-time computable, as formalized in Definition~\ref{def:pol_comp}. This holds whenever $f$ and $g$ are smooth and polynomial-time computable on $B_1$ and $\partial B_1$, respectively.

The results for the Laplace and diffusion equations were recently improved in \cite{bacho2023complexityblowupsolutionslaplace}, showing that the solution in each case is $\#P$-complete when the initial and boundary data are polynomial-time computable. Moreover, it was shown in \cite{bacho2023complexityblowupsolutionslaplace} that these equations exhibit a complexity blowup, implying that no numerical scheme can compute their solutions in polynomial-time for arbitrary initial and boundary data, unless $FP = \#P$.

More recently, a general result was established for linear PDEs of the form
\begin{equation} \label{eq:gen_lin_cb}   
\left\{
\begin{aligned}
    &\partial_t u = \mathcal{A}u, \textrm{ in }\,\,\Omega\times(0,1),\\&
    u|_{t=0} = \phi,\textrm{ in }\,\,\Omega,\\&
    \mathcal{L}u = 0, \textrm{ on }\,\,\partial\Omega\times(0,1),
\end{aligned} \right.
\end{equation}
where $\mathcal{A}$ is a linear differential operator, $\phi$ the initial condition and $\mathcal{L}$ is a linear operator describing the boundary conditions. In~\cite{Koswara2021-sr}, it is shown that the solution $u$ of~\eqref{eq:gen_lin_cb} is computable in $PSPACE$ \cite{Arora2012-bh} under certain conditions on the initial and boundary data. In particular, the differential operator $\mathcal{A}$ must be computable in polynomial-time, with the solution admitting a finite-difference approximation $u_n$ converging to the solution $u$ in the supremum norm.
 
The contributions of the present paper follow a similar philosophy to \cite{Koswara2021-sr}, where we estimate the complexity of a numerical scheme approximating the solution of the PDE to characterize the intrinsic complexity of the solution itself. Our approach, however, introduces a novel and more general framework based on discrete gradient flows, offering a powerful characterization of the complexity for broad classes of PDE solutions and identifying regimes in which their solutions exhibit complexity blowup.
\section{Contributions}
In this paper, we present a novel variational framework for analyzing the computability and complexity of different classes of PDE solutions by considering a variational formulation whose minimizer solves the PDE, and estimating the complexity of its discrete gradient flow. The theoretical framework developed here builds upon previous work on Sobolev cubatures \cite{Cardona2023-ma, Suarez2025} and polynomial surrogate models for PDE learning problems \cite{Cardona2024-sr}. The main contributions on this paper are summarized as follows:
\begin{itemize}
    \item \textbf{Sufficient conditions for computability}: We provide sufficient conditions under which the solution of a (possibly nonlinear) PDE is $H^k$-computable for some $k \in \mathbb{N}$, expressed in terms of structural properties of the variational losses (e.g., convexity, coercivity, continuity). Specifically, we present in Theorem~\ref{thm:compt} sufficient conditions for a loss of the form 
    \begin{equation}
        \mathcal{L}[u]:=\|\mathcal{N}[u]-f\|_{H^a(\Omega)}^2+\|u-g\|_{H^b(\partial\Omega)}^2,
    \end{equation}
    to have an \emph{$H^k$-computable solution} provided that the input data and its solution belong to the space $AC^{k-1}(\Omega)\cap BV^k(\Omega)$, for some $k\in \mathbb{N}$. 
    \item \textbf{Polynomial-time computability of PDE Solutions}: We prove in Theorem~\ref{thm:compt} that a class of PDEs has \emph{polynomial-time $H^k$-computable solutions}, provided the cubature loss is polynomial-time computable, the input data is sufficiently regular, and the solution operator preserves this regularity.
    \item  \textbf{Lower Bounds and Complexity Blow-Up}: We provide lower bounds for the complexity of non-analytic solutions in Corollary~\ref{cor:complexity}, complementing the results from Theorem~\ref{thm:compt}. 
    Consequently, we provide sufficient conditions for a PDE to exhibit \emph{complexity blowup} whenever the solution has a super-polynomial convergence rate. 
    \item \textbf{Illustrative PDE Examples – Polynomial-Time Computability and Complexity Blow-Up}: We present two examples demonstrating the applicability of our variational framework: a linear case with a polynomial-time computable solution and a nonlinear case exhibiting complexity blowup. Specifically, Theorem~\ref{thm:comp_poisson} shows that the Poisson equation admits \emph{$H^1$-computable} solutions with \emph{polynomial-time complexity} under additional regularity assumptions, whereas Theorem~\ref{thm:compl_blowup} demonstrates that the Eikonal equation has \emph{$L^2$-computable} solutions but exhibits \emph{complexity blowup} regardless of the regularity of the boundary data.

\end{itemize}
These results build a rigorous bridge between computability, complexity theory, and variational formulations of PDEs. By establishing upper and lower bounds on solution complexity within the adopted computability framework from Definition~\ref{def:comp_bsp}, the analysis determines when a PDE admits effective numerical schemes and when it exhibits complexity blowup. In particular, it links structural properties of the PDE losses—such as convexity and coercivity—and the regularity of the solution operator to the resulting complexity of the infinite-dimensional solution, relative to effective polynomial representations. Thus, the framework advances the theoretical foundations of computability and complexity for PDEs within this representation-dependent setting and provides constructive insights for designing effective numerical schemes.

\section{Variational Framework}
This section presents a variational framework that will allow us to characterize the complexity of solutions to a class of PDEs by leveraging the structural properties of their governing operators. 

For this, let $\Omega\subseteq \mathbb{R}^d$ be an open, bounded domain. We start by considering the following general PDE given by
\begin{equation}\label{eq:gen_PDE}
    \biggr\{\begin{array}{cc}
         \mathcal{N}[u] = f,& \textrm{ in }  \Omega,\\
         u|_{\partial\Omega} = g,& \textrm{ on }  \partial\Omega,
    \end{array}
\end{equation}
where $\mathcal{N}: X(\Omega)\rightarrow Y(\Omega)$ is a (possibly non-linear) differential operator, $f\in Y(\Omega)$ the right-hand side and $g\in W(\partial\Omega)$ the Dirichlet boundary conditions, with $X,Y,W$ some Banach spaces. Furthermore, we consider the following variational forward problem.
\begin{definition}\label{def:gen_pde_learning}
    Let $\mathcal{N}:X(\Omega)\rightarrow Y(\Omega)$ be an operator, $f\in Y(\Omega)$ the right-hand side and $g\in W(\partial\Omega)$ the Dirichlet boundary condition. The \emph{forward PDE problem} consists of finding a function $u^*\in X(\Omega)$ solving 
    \begin{equation}\label{eq:var_PDE}
    \inf\limits_{u\in X} \mathcal{L}[u]:= \|\mathcal{N}[u]-f\|_{Y(\Omega)}^2 + \|u|_{\partial\Omega}-g\|_{W(\partial\Omega)}^2.
\end{equation}
The functional $\mathcal{L}:X(\Omega)\rightarrow \mathbb{R}_{+}$ is called the \emph{forward problem loss}. 
\end{definition}
To solve the forward problem in \eqref{eq:var_PDE} we consider the gradient flow of the variational problem. 
\begin{definition}
    Let $\mathcal{L}:X(\Omega)\rightarrow \mathbb{R}_+$ be a loss functional as in \eqref{eq:var_PDE} defined over a Banach space $X(\Omega)$. The \emph{PDE gradient flow} is defined as the curve $\gamma_u:[0,\infty)\rightarrow X(
    \Omega
    )$ solving
    \begin{equation}\label{eq:GF-ID}
        \left\{\begin{aligned}
\dot{\gamma}_u(t) &= -\nabla \mathcal{L}[\gamma_u(t)], &\textrm{ for all } t \in (0, \infty), \\
\gamma_u(0) &= u_0,
\end{aligned}\right.
    \end{equation}
given some initial guess $u_0\in X(\Omega)$.
\end{definition}
To analyze the complexity of the minimizer $u^*\in X(\Omega)$, solving the variational problem in Definition~\ref{def:gen_pde_learning}, we need to consider three finite-dimensional approximations of the PDE forward problem, namely the surrogate approximation of the solution $u^*$, the cubature approximation of the loss $\mathcal{L}:X(\Omega)\rightarrow\mathbb{R}_+$ and a finite dimensional approximation of the gradient flow in \eqref{eq:GF-ID}. Following this, we introduce the notion of surrogate spaces, cubature approximations, and discrete gradient flows. 
\begin{definition}\label{def:surr_sp}
    A finite-dimensional space $\mathcal{S}_m(\Omega)\subseteq X(\Omega)$ is called a \emph{surrogate space} of degree $m\in \mathbb{N}$, approximating the Banach space $X(\Omega)$, if there exists a nested sequence of finite-dimensional spaces $(\mathcal{S}_m)_{m\in \mathbb{N}}\subseteq X(\Omega)$, such that $\mathcal{S}_{m}(\Omega)\subset \mathcal{S}_{m+1}$, for all $m\in \mathbb{N}$, and the union space
    \begin{equation}
        \mathcal{S}_{\infty}(\Omega):= \bigcup\limits_{m\in \mathbb{N}}\mathcal{S}_{m}(\Omega),
    \end{equation}
    is dense in $X(\Omega)$, i.e., $\overline{ \mathcal{S}_\infty(\Omega)}^{||\cdot||_X}= X(\Omega)$, with respect to the norm topology in $X(\Omega)$. We denote with $\hat{u}_\theta\in \mathcal{S}_m(\Omega)$ a surrogate function parametrized by the coefficients $\theta\in \Theta_m\cong\mathbb{R}^{d_m}$, with $d_m:=\mathrm{dim}(\mathcal{S}_m)$.
\end{definition}
Next, we define the notion of cubature of a loss. 
\begin{definition}\label{def:cub}
    Let $X(\Omega)$ and $Y(\Omega)$ be two Banach spaces, $\mathcal{F}:X(\Omega)\rightarrow Y(\Omega)$ a (possibly non-linear) operator, and $\mathcal{L}:X(\Omega)\rightarrow \mathbb{R}_+$ a functional of the form
    $$\mathcal{L}[u]:=\|\mathcal{F}[u]\|_{Y(\Omega)}^2.$$ 
    Furthermore, let $\mathcal{S}_n\subseteq X(\Omega)$ be a surrogate space, approximating $X(\Omega)$ in the strong topology and further assume that $X(\Omega)$ is a Hilbert space, then $\mathcal{P}_n:X(\Omega)\rightarrow \mathcal{S}_n(\Omega)$, $u\mapsto u_n\in \mathcal{S}_n(\Omega)$ is the orthogonal projection operator for a given $n\in \mathbb{N}$. We define the \emph{cubature approximation}  $\mathcal{L}^n:X(\Omega)\rightarrow \mathbb{R}_+$ of the loss $\mathcal{L}$ as
    \begin{equation}\label{eq:cub_loss}
        \mathcal{L}^n[u]:=\|\mathcal{F}[\mathcal{P}_n[u]]\|_{Y(\Omega)}^2.
    \end{equation}
\end{definition}
Using Definitions~\ref{def:surr_sp} and \ref{def:cub} we formulate the finite-dimensional PDE forward problem.
\begin{definition}
    Let $\mathcal{S}_m(\Omega)\subseteq X(\Omega)$ be a surrogate space and $\mathcal{L}^n:X(\Omega)\rightarrow\mathbb{R}_+$ a cubature approximation of the PDE forward loss $\mathcal{L}:X(\Omega)\rightarrow\mathbb{R}_+$. The \emph{finite-dimensional PDE forward problem} consists in finding an optimal surrogate $\hat{u}_{\theta^*}\in \mathcal{S}_m(\Omega)$ solving
    \begin{equation}
        \inf\limits_{\hat{u}_\theta\in \mathcal{S}_m(\Omega)} \mathcal{L}^n[\hat{u}_\theta].
    \end{equation}
\end{definition}
\begin{remark}
From the perspective of effective computability, solving boundary value problems in Sobolev spaces via variational formulations is, in general, not equivalent to solving boundary value problems with continuous boundary data \cite{Boche2026ComputableEnergy,Boche2026-kq}.
\end{remark}
Next, we define the Euler approximation of the PDE gradient flow.   
\begin{definition}
    Let $\mathcal{L}^n:X(\Omega)\rightarrow \mathbb{R}_+$ be the cubature approximation of the loss in \eqref{eq:var_PDE}, and $\theta^0\in \Theta_m$ some initial guess. We define the \emph{Euler gradient flow} $\gamma_\theta:[0,\infty)\rightarrow\Theta_m$ as the piecewise linear interpolation of the coefficients $\{\theta^j\}_{j=1}^{N_j}\subseteq\Theta_m$, $N_j\in \mathbb{N}$, given by the recurrence relation
    \begin{equation}
        \theta^{j+1} = \theta^j - \delta\tau\nabla\mathcal{L}^n[\hat{u}_{\hat{\theta}^{j}}],
    \end{equation}
    where $\delta \tau\in \mathbb{R}_{>0}$ is the time-step and $\hat{\theta}^{j}\in\{\theta^j,\theta^{j+1}\}$ depends on the choice of explicit ($\hat{\theta}^{j} =\theta^j$) or implicit ($\hat{\theta}^{j} =\theta^{j+1}$) Euler approximation of the gradient flow.
\end{definition}
Building on the various approximations introduced above, we present an error decomposition theorem for a class of loss functionals. To this end, we first recall several useful characterizations of these functionals.
\begin{definition}\label{def:RSI_cond}
    A loss functional $\mathcal{L}:H(\Omega)\rightarrow \mathbb{R}_+$, defined on a Hilbert space $H(\Omega)$, with minimizer $u^*\in H(\Omega)$, is said to satisfy the \emph{restricted secant inequality (RSI)} condition with modulus $\mu\in \mathbb{R}_+$, if for all $u\in H(\Omega)$ the following holds:
    \begin{equation}
        \langle \nabla\mathcal{L}[u],u-u^*\rangle_H\geq\mu\|u-u^*\|_{H(\Omega)}^2. 
    \end{equation}
\end{definition}
Next, we recall the definition of strongly convex functionals.
\begin{definition}\label{def:lam_conv}
    A loss functional $\mathcal{L}:H(\Omega)\rightarrow \mathbb{R}_+$, defined on a Hilbert space $H(\Omega)$, is said to be \emph{strongly convex} with modulus $\lambda\in \mathbb{R}_+$ (\emph{$\lambda$-convex}) if for every $u,v\in H(\Omega)$, the following holds 
    \begin{equation}
        \mathcal{L}[u]\geq\mathcal{L}[v] +\langle\nabla\mathcal{L}[v],u-v \rangle_{H(\Omega)} +\frac{\lambda}{2}\|u-v\|_{H(\Omega)}^2.
    \end{equation}
\end{definition}
Following Definition~\ref{def:lam_conv}, we present a specific example of a strongly convex functional.
\begin{lemma}\label{lemm:lam_convlin}
    Let $\mathcal{L}:H(\Omega)\rightarrow \mathbb{R}_+$ be a loss functional of the form \begin{equation}
        \mathcal{L}[u] := \|Au+b\|_{H(\Omega)}^2,
    \end{equation}
    with $u,b\in H(\Omega)$, and $A: H(\Omega)\rightarrow H(\Omega)$ a linear operator. Assume that $A$ is a coercive operator, i.e., 
    \begin{equation}
        \|Au\|_{H(\Omega)}^2\geq \mu \|u\|_{H(\Omega)}^2,
    \end{equation}
    for some $\mu\in\mathbb{R}_+$. Then, the loss $\mathcal{L}$ is $2\mu$-convex.
\end{lemma}
\vspace{-1em}\begin{proof}
  The proof of this standard result is included in Appendix~\ref{appendix_B}.
\end{proof}
Next, we recall an important condition characterizing the asymptotic growth behavior of functionals.
\begin{definition}
    A loss functional $\mathcal{L}:H(\Omega)\rightarrow \mathbb{R}_+$, defined on a Hilbert space $H$ with minimizer $u^*\in H(\Omega)$, is said to satisfy the \emph{quadratic growth condition (QGC)} with modulus $\mu\in \mathbb{R}_+$ if for every $u\in H(\Omega)$, the following holds: 
    \begin{equation}
        \mathcal{L}[u]-\mathcal{L}[u^*]\geq\mu \|u-u^*\|_{H(\Omega)}^2.
    \end{equation}
\end{definition}
The next result shows the equivalence between the QGC and the RSI condition for convex functionals.
\begin{lemma}\label{lemma:QGC}
     Let $\mathcal{L}:H(\Omega)\rightarrow \mathbb{R}_+$ be a convex loss functional, with a minimizer $u^*\in H(\Omega)$, satisfying the QGC with modulus $\mu\in \mathbb{R}_+$. Then $\mathcal{L}$ satisfies the RSI condition with the same modulus $\mu$. 
\end{lemma}
\vspace{-1em}\begin{proof}
   Let \( u^* \in H(\Omega) \) be a minimizer of \( \mathcal{L} \). Then, by convexity and the quadratic growth condition, we have that 
    \begin{equation}
        \langle \nabla\mathcal{L}[u], u-u^*\rangle_{H(\Omega)}\geq \mathcal{L}[u]-\mathcal{L}[u^*]\geq \mu\|u-u^*\|_{H(\Omega)}^2,
    \end{equation}
    for all \( u \in H(\Omega) \).
\end{proof}
Following this, we present the error decomposition theorem, which plays a central role in estimating the complexity of PDE solutions.
\begin{theorem}\label{thm:erro_dec}
 Let $\mathcal{S}_m(\Omega)\subseteq H(\Omega)$ be a surrogate space approximating a Hilbert space $H(\Omega)$, $\hat{u}_{\theta^j}\in \mathcal{S}_m(\Omega)$ a surrogate model obtained after $j$-iterations of an Euler gradient flow with time-step $\delta \tau\in \mathbb{R}_{> 0}$, and $\mathcal{L}^n:\mathcal{S}_m(\Omega)\rightarrow\mathbb{R}_+$ the cubature approximation of a loss $\mathcal{L}:\mathcal{S}_m(\Omega)\rightarrow\mathbb{R}_+$. If $\mathcal{L}$ satisfies the QGC with modulus $\mu\in \mathbb{R}_{>0}$, and has a minimizer $\hat{u}_{\theta^*}\in \mathcal{S}_m(\Omega)$, then the convergence error to the minimizer $u^*\in H(\Omega)$ of $\mathcal{L}$, can be estimated as
    \begin{equation}\label{Eq:erro_dec}
        \|u^*-\hat{u}_{\theta^j}\|_{H(\Omega)}^2\lesssim \epsilon_\mathrm{app}(m) + \epsilon_\mathrm{int}(n) + \epsilon_\mathrm{opt}(j).
    \end{equation}
     Here, the approximation, integration, and optimization errors are given by
    \begin{equation}
        \begin{array}{cc}
           &\epsilon_\mathrm{app}(m):=\|\hat{u}_{\theta^*}-u^*\|_{H(\Omega)}^2,\\
           &\epsilon_\mathrm{int}(n):= |\mathcal{L}[\hat{u}_{\theta^*}]-\mathcal{L}^n[\hat{u}_{\theta^*}]|,\\&
           \epsilon_\mathrm{opt}(j):= |\mathcal{L}^n[\hat{u}_{\theta^*}]-\mathcal{L}^n[\hat{u}_{\theta^j}]|.
        \end{array}
    \end{equation}
    Furthermore, if the map $\theta\mapsto \nabla\mathcal{L}^n[\hat{u}_\theta]$ is $L$-Lipschitz, then the optimization error of the explicit Euler gradient flow converges exponentially fast, i.e.,
    \begin{equation}\label{eq:opt_exp}
        \epsilon_\mathrm{opt}(j)=\mathcal{O}\left((1-\mu/L)^{j}\right).
    \end{equation}
\end{theorem}
\vspace{-1em}\begin{proof}
The proof of the error inequality in \eqref{Eq:erro_dec} follows directly from the quadratic growth condition together with the triangle inequality; see \cite{Suarez2025} for further details. The exponential convergence of the Euler gradient flow is then obtained by applying Lemma~\ref{lemma:QGC} in combination with the classic result for Lipschitz-smooth functionals satisfying the RSI condition \cite{Karimi2016-zl}.
\end{proof}
\begin{remark}\label{rk:lam_conv}
It is a standard result \cite{Karimi2016-zl} that the exponential decay of the optimization error in \eqref{eq:opt_exp} is guaranteed whenever the functional $\mathcal{L}:\mathcal{S}_m(\Omega)\to\mathbb{R}_+$ is $\lambda$-convex, since $\lambda$-convexity is a stronger condition than the QGC. 
\end{remark}
\section{Finite-Dimensional Approximations}
This section reviews the approximation-theoretical foundation for our framework for estimating the PDE solution's computational complexity. We examine the convergence rates at which polynomial spaces can approximate ground-truth solutions, as well as their associated loss functionals. These approximation rates are intrinsically linked to computability and complexity classes, as we elaborate on in Section~\ref{sec:compl_th}.

Let $\Omega := (-1,1)^d$ denote the $d$-dimensional hypercube. We begin by recalling the definition of the polynomial spaces $\Pi_{n}(\Omega) \subseteq C^{\infty}(\Omega)$ of $l^\infty$-degree $n \in \mathbb{N}$.
\begin{definition}\label{def:pol}
    Let $\{x^\alpha\}_{\alpha \in A_{n,d}}$ be the \emph{canonical basis} $x^\alpha:= \prod\limits_{i=1}^d x_i^{\alpha_i}$ for all $\alpha\in A_{n,d}$ where $A_{n,d}$ is the multi-index set $A_{n,d}:=\{\alpha \in \mathbb{N}^d:\|\alpha\|_{\infty}\leq n\}$, with $|A_{n,d}| = (n+1)^d$, ordered in \emph{lexicographic order} $\preceq$. A \emph{polynomial space of degree $n\in \mathbb{N}$} is defined as the span of the \emph{canonical basis} $\Pi_{n}:=\mathrm{span}\{x^\alpha\}_{\alpha\in A_{n,d}}$. We denote with $\Pi_{n}(U):=\{Q|_{U}:Q\in \Pi_{n}\}$, the restriction of the polynomial space $\Pi_n$ to a domain $U\subseteq \mathbb{R}^d$, and $\Pi_\infty(U):=\bigcup\limits_{n\in \mathbb{N}}\Pi_n(U)$ the union space.
\end{definition}
Next we recall the definition of the Chebyshev polynomials.
\begin{definition}
 The \emph{Chebyshev polynomials of the first kind}, denoted as $T_n\in \Pi_{n}$, are the polynomials satisfying the condition
    \begin{equation}
        T_n(\cos(\theta)) = \cos(n\theta), \textrm{ for all }\theta\in [0,2\pi]. 
    \end{equation}
    The \emph{$d$-dimensional Chebyshev polynomial of $l^\infty$-degree} is defined as
    \begin{equation}
        T_\alpha(x) := \prod\limits_{i=1}^{d}T_{\alpha_i}(x_i),
    \end{equation}
    for all $\alpha\in A_{n,d}$.
\end{definition}

\begin{theorem}[\cite{Trefethen2019}, Chapter 3]\label{thm:cheb_exp}
    Let $f:[-1,1]^{d}\rightarrow \mathbb{R}$ be a Lipschitz continuous function. Then, $f$ has a unique representation as an absolutely convergent Chebyshev series
    \begin{equation}
        f(x) = \lim\limits_{n\rightarrow \infty}\sum\limits_{\alpha\in A_{n,d}}\theta_\alpha T_\alpha(x).
    \end{equation}
    The coefficients are given for $\|\alpha\|_1\geq 1$, by
    \begin{equation}
        \theta_\alpha := \frac{2^{d}}{\pi^{d}}\int\limits_{[-1.1]^{d}}\omega_d(x)f(x)T_\alpha(x) dx,
    \end{equation}
    with the weight function $\omega_d(x):= \prod\limits_{i=1}^{d} (1-x_i^2)^{-1/2}$. We define the $n$-th degree \emph{Chebyshev projection} as $f_n(x):= \sum\limits_{\alpha \in A_{n,d}}\theta_\alpha T_\alpha(x)$.
\end{theorem}
\vspace{-1em}\begin{proof}
    A proof of this result can be found in the 1-dimensional setting in Chapter 3 of~\cite{Trefethen2019} and in higher dimensions in \cite{MASON1980349}.
\end{proof}
We can characterize the convergence of the Chebyshev series in terms of the regularity of the functions.

We start by defining the notion of $\rho$-analytic functions. 
\begin{definition}\label{assm:Berns_ell}
    Let $E_{\rho_i} \in \mathbb{C}$ be the $i$-th \emph{Bernstein Ellipse} \cite{Trefethen2017_mD}, defined as the ellipse in the complex  $x_i$-plane with foci $-1$ and $1$ and radius $\rho_i \in \mathbb{R}_+$ for $i\in \{1,\dots, d\}$. The \emph{elliptic polycylinder} $E(\rho)\subseteq \mathbb{C}^{d}$, $\rho\in\mathbb{R}^{d}$, is defined as 
    $$
    E(\rho):=\{z\in \mathbb{C}^{d}: z_i \in E_{\rho_i},\textrm{ for all }i\in \{1,\dots,d\}\}.
    $$
   A function $f$ is called $\rho$-analytic, if it admits an analytic extension in the \emph{elliptic polycylinder} $E(\rho)$,  with $\rho:=(\rho_1,\dots, \rho_{d})\in \mathbb{R}^{d}$, and $\rho_i>1$, for all $i\in \{1,\dots,d\}$.
\end{definition}
We have the following convergence result for $\rho$-analytic functions.
\begin{proposition}[\cite{Suarez2025}]\label{prop:conv_an}
    Let $f:\Omega\rightarrow \mathbb{R}$ be a $\rho$-analytic, bounded function, i.e., $\|f\|_{C^0(\Omega)}\leq M_f$ for some $M_f\in \mathbb{R}_+$. Then for each $n\in \mathbb{N} $, the error of the Chebyshev projection can be estimated as
    \begin{equation}
    \|f-f_n\|_{C^0(\Omega)}\lesssim C(d)M_f\frac{\rho^{-n}_*}{\rho_*-1},
    \end{equation}
    with $
    \rho_*:=\min\limits \{\rho_1,\dots,\rho_{d}\}$ and $C(d):=|\Omega|$.
    Furthermore, the $\beta$-derivative, with $\beta\in \mathbb{N}^d$, of the projection converges with the rate 
    \begin{equation}        
    \|\partial^\beta f-\partial^\beta f_n\|_{C^0(\Omega)}\lesssim  C(d)M_fn^{2|\beta|}\frac{\rho^{-n}_*}{\rho_*-1}.
    \end{equation}
\end{proposition}
In contrast, if the function has weaker regularity, we have the following convergence result.
\begin{proposition}[\cite{zavalani2024}]\label{prop:abs_cont}
    Let $f\in AC^{k-1}(\Omega)\cap BV^k(\Omega)$ be a function with $k$-th order bounded variations and $k-1$ absolutely continuous derivatives, then the Chebyshev approximation $f_n\in \Pi_{n}(\Omega)$ converges with the rate
    \begin{equation}
        \|f-f_n\|_{C^0(\Omega)}\lesssim C(d) \frac{V_k}{k(n-k)^{k}},
    \end{equation}
    where $V_k\in \mathbb{R}_{+}$ is the $k$-th order variation of $f$ and $C(d)\in \mathbb{R}_+$ a constant depending on the dimension and the domain. Furthermore, the $\beta$-derivative of the approximation can be estimated as
    \begin{equation}
        \|\partial^\beta f-\partial^\beta f_n\|_{C^0(\Omega)}\lesssim C(d) n^{2|\beta|}\frac{V_k}{k(n-k)^{k}}.
    \end{equation}
\end{proposition}
\vspace{-1em}\begin{proof}
    The proof of Propositions~\ref{prop:conv_an} and ~\ref{prop:abs_cont} can be found in 1-D in \cite{Trefethen2017_mD} and in general dimensions in \cite{Suarez2025} and \cite{zavalani2024}.
\end{proof}
In Section~\ref{sec:compl_th} we will see how Proposition~\ref{prop:abs_cont} is not sufficient for showing the complexity blowup of PDEs, as it only provides an upper bound on the convergence rate. To address this we recall a reverse theorem, following~\cite{Trefethen2019, Alimov2019-np}. 
\begin{theorem}[\cite{Trefethen2019}, Chapter 8]\label{thm:blowup}
    Let $f:\Omega\rightarrow \mathbb{R}$ be a function for which there exists a polynomial $Q_n\in \Pi_{n}(\Omega)$ satisfying
    \begin{equation}
        \|Q_n-f\|_{C^0(\Omega)}\leq M\rho ^{-n},
    \end{equation}
    for some constants $\rho\in \mathbb{R}_+$, and $M\in \mathbb{R}_+$. Then $f$ can be extended analytically to the Elliptic polycylinder $E(\rho)\subseteq \mathbb{C}^{d}$.
\end{theorem}
As a direct consequence of Theorem~\ref{thm:blowup}, we obtain that the convergence of non-analytic functions is sub-exponential. 
\begin{corollary}\label{cor:onlyif_an}
    If $f\in C^0(\Omega)$ is not a $\rho$-analytic function, then the convergence rate is, at best, sub-exponential, i.e., there exists a $\mathcal{R}_{\mathrm{sub}}:\mathbb{N}\rightarrow\mathbb{R}$ such that 
    \begin{equation}
        \|Q_n-f\|_{C^0(\Omega)}\geq\mathcal{R}_{\mathrm{sub}}(n),\textrm{ for all } n\in \mathbb{N,}
    \end{equation}
    with $\mathcal{R}_{\mathrm{sub}}(n) =\mathcal{O}(e^{-n^\alpha})$ for $\alpha\in (0,1)$.
\end{corollary}

Next, we introduce the Lagrange polynomials.
\begin{definition}\label{def:Lag_p}
Let $P_{n}(\Omega)\subseteq \Omega$ be the Legendre points \cite{GW-Method} re-scaled to $\Omega$. The \emph{Lagrange polynomials} ~\cite{Trefethen2019}, Chapter 17, $\{L_{\alpha}\}_{\alpha\in A_{n,d}}\subseteq \Pi_{n}(\Omega)$ defined with respect to the Legendre grid, are given by
\begin{equation}\label{eq:Lag}
    L_{\alpha}(x,t) = \prod_{i=1}^{d}l_{\alpha_i}(x_i)\,, \quad l_{\alpha_i}(x_i) = \prod_{\stackrel{p\in P_{n}(\Omega_i)}{p \not = p_{\alpha_i}}} \frac{x_i -p}{p_{\alpha_i}-p}\,,
    \end{equation}
    where $P_n(\Omega_i)$ is the 1-D Legendre grid rescaled to the projection of $\Omega$ in the $x_i$ direction, for all $i\in\{1,\dots ,d\}$.
\end{definition}
\begin{remark}\label{rk:interpolation}
Let $h\in C^0(\Omega)$, and $\mathcal{I}_n[h]\in \Pi_n(\Omega)$ its Lagrange interpolation given by the interpolation operator $\mathcal{I}_n:C^0(\Omega)\rightarrow \Pi_n(\Omega)$ satisfying
\begin{equation}
    \mathcal{I}_n[h](p_\alpha) = h(p_\alpha),\textrm{ for all }p_\alpha\in P_n(\Omega).
\end{equation}
  By definition of the Lagrange basis, the coefficient operator $C_n:C^0(\Omega)\rightarrow \mathbb{R}^{|A_{n,d}|}$ of the Lagrange interpolation is explicitly given by \begin{equation}
      C_n[h]:= h(P_n(\Omega)).
  \end{equation}
  Consequently, $\mathcal{I}_n[h](x) = \sum\limits_{\alpha\in A_{n,d}}(\mathcal{C}_n[h])_\alpha L_\alpha (x)$, for all $x\in \Omega$ and all $h\in C^0(\Omega)$.
\end{remark}
To approximate the PDE losses in \eqref{eq:var_PDE}, we introduce the Sobolev cubatures following \cite{Cardona2024-sr,Cardona2023-ma}.
\begin{definition}
    Let $\mathbb{D}^\beta\in \mathbb{R}^{|A_{n,d}|\times |A_{n,d}|}$ be the polynomial differentiation matrix defined such that for all $Q\in \Pi_n(\Omega)$ we have 
    \begin{equation}
        \partial^\beta Q =\sum\limits_{\alpha\in A_{n,d}}(\mathbb{D}^\beta \mathfrak{q})_\alpha L_{\alpha}(x),
    \end{equation}
    with $\beta\in \mathbb{N}^d$ and $\mathfrak{q}:=Q(P_{n}(\Omega))\in  \mathbb{R}^{|A_{n,d}|}$. Following this, we can define the \emph{Sobolev cubature} of positive order $k\in \mathbb{N}$ and degree $n\in \mathbb{N}$, approximating the $H^k$-norm, as
    \begin{equation}
          I_n^k[u]:= \|\mathcal{I}_n[u]\|_{H^k(\Omega)}^2=\mathfrak{u}^T\mathbb{W}_k\mathfrak{u},
    \end{equation}
    with $\mathfrak{u}:=u(P_n(\Omega))$, $\mathbb{W}:=\mathrm{diag}(\{\omega_\alpha\}_{\alpha\in A_{n,d}})$ the Legendre weights matrix with $\omega_\alpha:=||L_\alpha||_{L^2(\Omega)}^2$ for all $\alpha\in A_{n,d}$, and $\mathbb{W}_k:=\sum\limits_{|\beta|\leq k}(\mathbb{D}^{\beta})^T\mathbb{W}\mathbb{D}^{\beta}$ the Sobolev cubature weights. 
    \end{definition}
The Sobolev cubatures inherit the approximation properties of the Chebyshev series by using that $I_n^k[Q] = \|f_m\|_{H^k(\Omega)}^2$ for all $f_m\in \Pi_{m}(\Omega)$ with $m\leq 2n+1$, i.e., the Sobolev cubatures are exact for polynomials of up to degree $2n+1$. The following convergence rates hold for the Sobolev cubatures.
\begin{theorem}\cite{Suarez2025}\label{thm:sob_cub_s}
        Let $f\in H^k(\Omega)$ be a Sobolev function. If in addition $f\in AC^{k-1}(\Omega)$, then the $s$-th order Sobolev cubature of degree $n\in \mathbb{N}$ converges to the squared Sobolev norm $I^s[f]:=\|f\|_{H^s(\Omega)}^2$, for $0\leq s\leq k$, with the rate
    \begin{equation}
        |I^s[f]-I^s_n[f]|\lesssim n^{2(s-k)}\|f\|_{H^{k}(\Omega)}^2+C(d)\frac{ n^{2s}(V_{k})^2}{ k (n-k)^k},
    \end{equation}
     where $V_{k}:=\mathcal{V}_k(f)$ are the $k$-th order variations of $f$ \cite{Cardona2024-sr}.

     Moreover, if $f$ is $\rho$-analytic in the elliptic polycylinder $E(\rho)$, then the $s$-th order Sobolev cubature of degree $n\in \mathbb{N}$ converges to the squared Sobolev norm $I^s[f]$, with the rate
    \begin{equation}
        |I^s[f]-I^s_n[f]|\lesssim C(d)(M_{f})^2\frac{n^{2s}\rho_*^{-n}}{(\rho_*-1)},
    \end{equation}
    where $\rho_*:=\min\limits \{\rho_1,\dots,\rho_{d}\}$ and $M_f:=\|f\|_{C^0(\Omega)}<+\infty$.
\end{theorem}
\vspace{-1em}\begin{proof}
    The proof of this result can be found in \cite{Cardona2024-sr}.
\end{proof}
\section{Computability Theory in Banach Spaces}\label{sec:compl_th}
In this section, we introduce the notion of Turing computability in Banach spaces, following Chapter 2 of \cite{Pour-El1989-yg}, and discuss some of its analytical properties. Our focus centers on computability in the space of continuous functions $C^0(\Omega)$ and the Sobolev spaces $H^k(\Omega)$ of order $k \in \mathbb{R}$. Let $\Omega:=(-1,1)^d\subseteq \mathbb{R}^d$ be the computable hypercube of dimension $d$, We begin by revisiting the concept of $C^0$-computability.
\begin{definition}[Effective Weierstrass]\label{def:C_0_comp}
    A function $f:\Omega\rightarrow \mathbb{R}$ is \emph{$C^0$-computable}, if there exists a computable sequence of polynomials $(p_n)_{n\in \mathbb{N}}\subseteq \Pi_{n}(\Omega)$ converging \emph{effectively} to $f$ in the supremum norm. More precisely, there exists a recursive function $e:\mathbb{N}\rightarrow\mathbb{N}$ such that for all $N\in \mathbb{N}$
    \begin{equation}
        n\geq e(N) \Rightarrow \|f-p_n\|_{C^0(\Omega)}\leq 2^{-N}.
    \end{equation}
\end{definition}
Following this, we can define computable functions in general Banach spaces.
\begin{definition}\label{def:comp_bsp}
    Let $X(\Omega)$ be a Banach space  defined as the completion $X(\Omega)=\overline{C^0(\Omega)}^{||\cdot||_X}$, with a computability structure as introduced in Definition~\ref{def:comp_str}. A function $f\in X(\Omega)$ is \emph{$X$-computable} if there exists a sequence $(g_n)_{n\in \mathbb{N}}\subseteq C^0(\Omega)$ of $C^0$-computable functions and a recursive function $e:\mathbb{N}\rightarrow\mathbb{N}$ such that for all $N\in \mathbb{N}$,
    \begin{equation}
        n\geq e(N) \Rightarrow \|f-g_n\|_{X}\leq 2^{-N}.
    \end{equation}
    Moreover, let $(f_n)_{n\in \mathbb{N}}\subseteq X$ be a sequence in $X$. It is \emph{$X$-computable}, if there exists a $C^0$-computable double sequence $(g_{n,k})_{n,k\in \mathbb{N}}\subseteq C^0(\Omega)$ such that $\|f_{n}-g_{n,k}\|_X\stackrel{k\rightarrow\infty}{\longrightarrow 0}$ effectively in $n$ and $k$.
\end{definition}
Before defining the complexity classes considered in this work, we introduce the notions of residual and cost functions. 
\begin{definition}
    The \emph{exponential, sub-exponential and polynomial residual functions} $\mathcal{R}_\mathrm{exp},\mathcal{R}_\mathrm{sub},\mathcal{R}_\mathrm{pol}:\mathbb{N}\rightarrow\mathbb{R}_+$, are asymptotically decaying functions satisfying
    \begin{equation}
    \mathcal{R}_\mathrm{exp}(n) =\mathcal{O}( 2^{-n}), 
    \,\,\, 
    \mathcal{R}_\mathrm{sub}(n) = \mathcal{O}(  2^{-n^\alpha}), 
    \,\,\, 
    \mathcal{R}_\mathrm{pol}(n)=\mathcal{O}(  n^{-k}),
\end{equation}
for some $0<\alpha<1$, and $k\in \mathbb{N}$, as $n\rightarrow\infty$.

On the other side, the \emph{super-polynomial and polynomial cost functions} $\mathcal{C}_\mathrm{sup},\mathcal{C}_\mathrm{pol}:\mathbb{N}\rightarrow\mathbb{R}_+$, are asymptotically growing functions satisfying
 \begin{equation}
    N^s= \mathcal{O}(\mathcal{C}_\mathrm{sup}(N)), 
    \,\,\, 
    \mathcal{C}_\mathrm{pol}(N) =\mathcal{O}(N^{k}),
\end{equation}
for all $s\in \mathbb{N}$, and some $k\in \mathbb{N}$, as $N\rightarrow\infty$.
\end{definition}
\begin{remark}
For sufficiently large $n \in \mathbb{N}$, the residual functions satisfy 
\begin{equation}
    \mathcal{R}_{\mathrm{exp}}(n) < \mathcal{R}_{\mathrm{sub}}(n) < \mathcal{R}_{\mathrm{pol}}(n).
\end{equation}
Moreover, for sufficiently large $N \in \mathbb{N}$, the cost functions satisfy
\begin{equation}
    \mathcal{C}_{\mathrm{sup}}(N) > \mathcal{C}_{\mathrm{pol}}(N).
\end{equation}
The hierarchies induced by the residual and cost functions provide a framework to characterize the complexity blow-up of functions, as discussed in Remark~\ref{rk:complexity_blowup}.
\end{remark}
Following this, we introduce the two complexity classes considered in this work.
Following this, we present the definition of the two complexity classes considered in this document. 
\begin{definition}\label{def:pol_comp}
Let $X(\Omega)$ be a Banach space defined as the completion $X(\Omega)=\overline{C^0(\Omega)}^{||\cdot||_X}$, with a computable structure as introduced in Definition~\ref{def:comp_str}. We say that a function $f\in X(\Omega)$ is \emph{$X$-computable in polynomial-time} if there exists a sequence $(q_n)_{n\in \mathbb{N}}\subseteq \Pi_\infty(\Omega)$ of  $C^0$-computable polynomials with polynomial-time computable coefficients, and an \emph{exponential} residual function  $\mathcal{R}_\mathrm{exp}:\mathbb{N}\rightarrow \mathbb{R}_+$, such that 
\begin{equation}\label{eq:conv_upper_bnd}
    \|q_n-f\|_X^2\leq \mathcal{R}_\mathrm{exp}(n),\textrm{ for all } n\in \mathbb{N}.
\end{equation}
Equivalently, the function $f$ can be approximated with precision $2^{-N}$ in $\mathcal{O}(\mathcal{C}_\mathrm{pol}(N))$ operations. 

Following this, we say that the function is $X$-computable in super-polynomial-time, if it is computable in $X$ and there exists a sub-exponential residual function $\mathcal{R}_{\mathrm{sub}}:\mathbb{N}\rightarrow\mathbb{R}_+$, such that 
\begin{equation}
 \|q_n-f\|_X^2\geq \mathcal{R}_{\mathrm{sub}}(n) ,\textrm{ for all } n\geq \mathbb{N}.
\end{equation}

Equivalently, the function $f$ can be approximated with precision $2^{-N}$ in at least $\mathcal{C}_\mathrm{sup}(N)$ operations.
\end{definition}
\begin{remark}
    Note that the polynomial-time characterization from Definition~\ref{def:pol_comp} requires the sequence of approximating polynomials to have polynomial-time computable coefficients. This is a necessary condition, as there exist sequences of polynomials whose coefficients are not computable in polynomial-time, as shown in~\cite{Boche2026-kq}.  

   Moreover, this condition ensures that the $n$-th approximating polynomial $q_n \in \Pi_n(\Omega)$ can itself be computed in polynomial time. Indeed, $q_n$ can be expressed as a linear combination of polynomially many Chebyshev basis functions $\{T_\alpha\}_{\alpha \in A_{n,d}}$, where $|A_{n,d}| = (n+1)^d$.

Consequently, the total cost of computing $q_n$ satisfies
\begin{equation}
    \mathcal{C}_{q_n}(n) = \mathcal{O}\big(C_q\cdot(n+1)^d\big),
\end{equation}
where $C_q\in \mathbb{N}$ denotes the maximal (polynomial) cost of computing the product of a single coefficient and its corresponding basis function.

In particular, if the residual converges exponentially fast as in~\eqref{eq:conv_upper_bnd}, then achieving an accuracy of order $2^{-N}$ requires a polynomial number of operations. Indeed, since $n = \mathcal{O}(N)$ suffices, we obtain
\begin{equation}
    \mathcal{C}_{q_n}(N) = \mathcal{O}\big(C_q\cdot(N+1)^d\big),
\end{equation}
which yields polynomial complexity.
\end{remark}
The next result presents a characterization of Banach spaces with a computability structure. 
\begin{theorem}\label{thm:compt_structures}
    Let $X$ be a Banach space such that the space of continuous functions $C^0\subseteq X$ is effectively dense in $X$, i.e., for every $f\in X$ there exists a sequence $(q_n)_{n\in\mathbb{N}}\subseteq C^0$ converging effectively to $f$. Then, the space $X$ has a computability structure.
\end{theorem}
\vspace{-1em}\begin{proof}
    The proof is given in the Appendix~\ref{appendix_A}.
\end{proof}
\begin{remark}
Using the standard result that $C^0_c(\Omega)$ spaces are dense in $L^2(\Omega)$ \cite{Rudin1986-qe}, it follows that Sobolev spaces are a natural example of Banach spaces equipped with a computability structure. 
\end{remark}
Following the definitions of computability presented in this section, we can establish the computability and complexity of the Sobolev cubature approximation of the loss functionals and their corresponding Euler updates.
\begin{proposition}\label{prop:comp_cub}
Let $f\in H^k(\Omega)$ be an $H^k$-computable function, then the $k$-th order Sobolev cubature
$$
I^k_n[f]=\|\mathcal{I}_n[f]\|_{H^k(\Omega)}^2,
$$
is computable. Furthermore, if $f$ is an $H^k$-computable function in polynomial-time, then $I^k_n[f]$ is computable in polynomial-time.
\end{proposition}
\vspace{-1em}\begin{proof}
    The proof of Proposition~\ref{prop:comp_cub} can be found in Appendix~\ref{appendix_A}.
\end{proof}
Following Proposition~\ref{prop:comp_cub} we can characterize the computability of the explicit Euler gradient flow. 
\begin{corollary}\label{cor:comp_GF}
    Let $\mathcal{L}^n:C^0(\Omega)\rightarrow \mathbb{R}_+$ be a loss of the form
    \begin{equation}
        \mathcal{L}^n[u]:=\|\mathcal{F}[u_n]-f_n\|_{H^k(\Omega)}^2,
    \end{equation}
    with $f_n\in \Pi_{n}(\Omega)$ a computable polynomial, and $\mathcal{F}:\Pi_{n}(\Omega)\rightarrow \Pi_{n}(\Omega)$ a differentiable operator with $\mathcal{F}[u_n]$ and $\nabla\mathcal{F}[u_n]$ being $C^0$-computable functions. Then, the $k$-th iteration of the explicit Euler gradient flow is $C^0$-computable.  
\end{corollary}
\vspace{-1em}\begin{proof}
    The full gradient of the loss is given by
    \begin{equation}
    \begin{aligned}
        \nabla\mathcal{L}^n[u] &= 2\sum\limits_{|\beta|\leq k}\nabla\mathfrak{F}(u)^T(\mathbb{D^\beta})^T\mathbb{W}\mathbb{D^\beta}(\mathfrak{F}(u)-f_n)\\&
        = 2\nabla\mathfrak{F}(u)^T\mathbb{W}_k(\mathfrak{F}(u)-f_n),
    \end{aligned}
    \end{equation}
    where $\nabla\mathfrak{F}(u):=\nabla\mathcal{F}[u](P_{n}(\Omega))$ and $\mathfrak{F}(u):=\mathcal{F}[u](P_{n}(\Omega))$. By the assumptions on $\mathcal{F}$, both quantities are computable. The result then follows from Proposition~\ref{prop:comp_cub}, noting that since the Sobolev cubature matrices are computable, so is the operator $\mathbb{W}_k = \sum\limits_{|\beta|\leq k}(\mathbb{D}^\beta)^\top \mathbb{W} \mathbb{D}^\beta$. And hence, the iterates 
    $$
   u^{j+1} =u^j-\delta\tau \nabla\mathcal{L}^n[u^{j}],
    $$
    with a computable $\delta\tau>0$, are computable. 
\end{proof}
\subsection{Complexity of PDE Solutions}
Building on Definition~\ref{def:comp_bsp}, the error decomposition in Theorem~\ref{thm:erro_dec}, and the results of Propositions~\ref{prop:conv_an} and \ref{prop:abs_cont}, we present the main contributions of this work. We characterize the computational complexity of solutions to a class of PDEs in terms of the structural properties of their associated variational loss. In particular, we establish sufficient conditions for $H^k$-computability, identify criteria guaranteeing polynomial-time computability, and outline regimes in which the solution undergoes a complexity blowup. We start by defining the PDE forward problem, and a formal characterization of complexity blowup. 
\begin{definition}\label{def:gen_PDE_loss}
    We define the PDE forward loss $\mathcal{L}: H^{k}(\Omega)\rightarrow\mathbb{R}_+$, $k\in \mathbb{N}$, as
\begin{equation}\label{Eq:PDE_loss}
    \mathcal{L}[u]:=\|\mathcal{N}[u]-f\|_{H^a(\Omega)}^2+\|u-g\|_{H^{b}(\partial\Omega)}^2,\textrm{ for all } u\in H^k(\Omega),
\end{equation}
where $\mathcal{N}:H^k(\Omega)\rightarrow H^a(\Omega)$ is a (possibly non-linear) operator, $g\in AC^{k'-3/2}(\partial\Omega)\cap H^{k'-1/2}(\partial\Omega),\,\,f\in AC^{k'-1}(\Omega)\cap H^{k'}(\Omega),$ 
for some $a,b\in \mathbb{Z}$, satisfying $a\leq k/2$ and $b\leq k/2-1/4$, and some $k'>\max\{k,k+2a, k+2b\}$.  
\end{definition}
Building on the notions of computability and complexity classes on Banach spaces introduced in Definitions~\ref{def:comp_bsp} and~\ref{def:complexity_blowup}, we define \emph{complexity blowup} in the context PDE forward problems.
\begin{definition}\label{def:complexity_blowup}
    Let $\mathcal{L}: H^{k}(\Omega)\rightarrow\mathbb{R}_+$, $k\in \mathbb{N}$, be a PDE loss defined as in Eq~\ref{Eq:PDE_loss}, with boundary condition $g\in H^{k-1/2}(\partial\Omega),$ and source term $f\in H^{a}(\Omega)$ with $a\in \mathbb{N}$ satisfying the condition given in Definition~\ref{def:gen_PDE_loss}. Given $f$ and $g$, let $u^*\in H^k(\Omega)$ be the corresponding minimizer of the loss $\mathcal{L}$. We say that the solution $u^*$ exhibits \emph{complexity blowup} if the input data (i.e., $f$ and $g$) is computable in polynomial-time, while  $u^*$ is computable in at least super-polynomial-time. 
\end{definition}
\begin{remark}\label{rk:complexity_blowup}
    Following Definition~\ref{def:pol_comp}, a PDE solution $u^*\in H^k(\Omega)$ exhibits complexity blowup, as formalized in Definition~\ref{def:complexity_blowup}, if it is $H^k$-computable, and there exists a residual polynomial function $\mathcal{R}_{\mathrm{pol}}:\mathbb{N}\times\mathbb{N}\rightarrow\mathbb{N}$ and a sub-exponential residual function $\mathcal{R}_{\mathrm{sub}}:\mathbb{N}\times \mathbb{N}\rightarrow\mathbb{N}$ such that the $j$-th iterate $\hat{u}_{\theta^j}\in \mathcal{S}_m(\Omega)$ of a discrete gradient flow is estimated as   
    \begin{equation}\label{eq:compl_bounds}
       \mathcal{R}_{\mathrm{sub}}(j,\mathrm{dim}(\mathcal{S}_m))\leq ||\hat{u}_{\theta^j}-u^*||_{H^k(\Omega)}^2\leq \mathcal{R}_{\mathrm{pol}}(j,\mathrm{dim}(\mathcal{S}_m)).
    \end{equation}
    Here $u^*\in H^k(\Omega)$ denotes the minimizer of the loss $\mathcal{L}$. 
\end{remark}
Next, we define a set of assumptions regarding the PDE forward loss.
\begin{assumption}\label{ass:PDE_loss}
    Let $\mathcal{L}:H^k(\Omega)\to\mathbb{R}+$ be a PDE loss. We assume that $\mathcal{L}$ is a differentiable, convex, coercive functional satisfying the QGC with modulus $\mu\in \mathbb{R}_{>0}$ at a minimizer $u^*\in H^k(\Omega)$. Moreover, for every $u_n \in \Pi_{n}(\Omega)$, both $\mathcal{N}[u_n]$ and $\nabla\mathcal{N}[u_n]$ are $C^0$-computable. Finally, let $\mathcal{L}^n:\mathcal{S}_m(\Omega)\rightarrow\mathbb{R}_+$ denote the Sobolev cubature approximation of $\mathcal{L}$, we assume that the map $\theta \mapsto\nabla \mathcal{L}^n[\hat{u}_\theta]$ is $L$-Lipschitz continuous for all $\hat{u}_\theta\in \mathcal{S}_m(\Omega)$. 
\end{assumption}
Following this, we provide the main computability theorem for a class of PDEs. 
\begin{theorem}\label{thm:compt}
Let $\mathcal{L}:\mathcal{S}_n(\Omega)\rightarrow \mathbb{R}_+$ be a PDE loss, as in Definition~\ref{def:gen_PDE_loss}, satisfying Assumption~\ref{ass:PDE_loss}, with $\mathcal{S}_{n}(\Omega):=\Pi_n(\Omega)$ as the space of polynomial of degree $n\in \mathbb{N}$. If, in addition, we assume that the minimizer $u^*$ satisfies $u^*\in AC^{k'-1}(\Omega)\cap BV^{k'}(\Omega),$ 
then $u^*$ is an $H^k$-computable solution of the PDE in \eqref{eq:gen_PDE} almost everywhere in $\Omega$. 
\end{theorem}
To prove Theorem~\ref{thm:compt}, we rely on the definition of the reconstruction problem together with a computability result for its solution.

\begin{definition}\label{Def:rec_prob}
    Let $h \in H^s(\mathcal{B})$ be a Sobolev function in $H^s(\mathcal{B})$, for some $s\in \mathbb{R}$ over a domain $\mathcal{B} \subseteq \mathbb{R}^{N_B}$, with $N_B \in \mathbb{N}$.  
    The \emph{reconstruction problem} with polynomial surrogates consists of finding the optimal $\hat{u}_\theta \in \Pi_{n}(\mathcal{B})$ that minimizes  
    \begin{equation}
        \mathcal{L}_r[\hat{u}_\theta] := \| \hat{u}_\theta - h \|_{H^s(\mathcal{B})}^2 .
    \end{equation}
\end{definition}
Subsequently, we establish sufficient conditions under which the reconstruction problem can be solved with polynomial complexity.
\begin{lemma}\label{lemm:comp_initial_data}
    Let $h\in H^k(\mathcal{B})$ be a Sobolev function. If, in addition $h\in AC^{k-1}(\mathcal{B})\cap BV^{k}(\mathcal{B})$, then $h$ is $H^k$-computable for every $0\leq s< k/2$. Furthermore, if $h$ is $\rho$-analytic, then the function is $H^k$-computable in polynomial-time.
\end{lemma}
\vspace{-1em}\begin{proof}
    The proof of Lemma~\ref{lemm:comp_initial_data} is included in Appendix~\ref{appendix_B}.
\end{proof}
\vspace{-1em}\begin{proof}[Proof of Theorem~\ref{thm:compt}]
    We start by noting that since the functional $\mathcal{L}:H^{k}(\Omega)\rightarrow \mathbb{R}_+$ is lower semi-continuous, convex, and coercive, hence the existence of a minimizer $u^*\in H^{k}(\Omega)$ is guaranteed by applying Tonelli's direct method \cite{Maso1993}. 

    Let $\hat{u}_{\theta^j}\in\Pi_{n}(\Omega)$ be a polynomial surrogate model obtained after $j$-iterations of the explicit Euler gradient flow of the Sobolev cubature approximation $\mathcal{L}^n:\Pi_n(\Omega)\rightarrow\mathbb{R}_+$ of the loss in \eqref{Eq:PDE_loss}. Since $\mathcal{L}$ satisfies the QGC, by Theorem~\ref{thm:erro_dec} we have 
    \begin{equation}\label{eq:err_dec}
        \|\hat{u}_{\theta^j}-u^*\|_{H^k(\Omega)}^2\leq \epsilon_{\mathrm{app}}(n) + \epsilon_{\mathrm{int}}(n)+\epsilon_{\mathrm{opt}}(j).
    \end{equation}

    Moreover, since $\mathcal{L}$ is convex and satisfies the QGC, the map $\theta\mapsto \mathcal{L}^n[\hat{u}_\theta]$ is convex and satisfies the RSI condition by Lemma~\ref{lemma:QGC}, with modulus $\mu\sigma \in\mathbb{R}_+$, where $\sigma$ is the smallest eigenvalue of the matrix $\mathbb{V}^T\mathbb{W}_k\mathbb{V}\in \mathbb{R}^{|A_{n,d}|\times|A_{n,d}|}$ as in the proof of Lemma~\ref{lemm:comp_initial_data}. Therefore, by Theorem~\ref{thm:erro_dec}, the optimization error converges with the rate
    \begin{equation}
        \epsilon_{\mathrm{opt}}(j) = \mathcal{O}\left((1-\mu\sigma/L)^{j}\right),
    \end{equation}
where we used the Lipschitz continuity of $\theta \mapsto \nabla\mathcal{L}^n[\hat{u}_\theta]$.
Moreover, by Propositions~\ref{prop:abs_cont} and the regularity assumptions on $f$ and $g$, we know that 
$$
\epsilon_{\mathrm{app}}(n) + \epsilon_{\mathrm{int}}(n) = \mathcal{O}(n^{-k'+k}+n^{-k^*}),
$$
where $k^*:= \min\{k'-k-2a, k'-k-2b\}$.
To conclude the proof, we apply Lemma~\ref{lemm:comp_initial_data} to ensure the computability of $f$ and $g$, and Corollary~\ref{cor:comp_GF} to establish the computability of the explicit Euler gradient flow. This implies that for all $n\geq e_n(N)$ and $j\geq e_j(N)$, where $e_n(N):=\max\{2^{N/(k-k')},2^{N/k^*}\}$ and $e_j(N) =N\log(1-\mu\sigma/L) $, the convergence error achieves fast convergence, i.e.,
\begin{equation}
    \|\hat{u}_{\theta^j}-u^*\|_{H^k(\Omega)}^2\leq 2^{-N}.
\end{equation}
Equivalently, in order to obtain an error of order $2^{-N}$ we need to compute $j=\mathcal{O}(N\log(1-\mu\sigma/L))$ iterations with $\mathcal{O}(F(\max\{2^{dN/(k-k')},2^{dN/k^*}\}))$ computations each, where $F:\mathbb{N}\rightarrow\mathbb{N}$ is a function proportional to the complexity of $\theta^j-\delta\tau\nabla\mathcal{L}^n[\hat{u}_{\theta^j}]$. 
\end{proof}
Note that Theorem~\ref{thm:compt} provides only a qualitative characterization of the complexity of the solution to the PDE forward problem. The following corollary refines this by linking the complexity class of the solution with its regularity.
\begin{corollary}\label{cor:complexity}
   Let the assumptions of Theorem~\ref{thm:compt} hold, and suppose that the source term $f$ and the boundary condition $g$ are $\rho$-analytic functions. Assume further that the PDE solution operator is regularity-preserving, i.e., $u^*$ is $\rho^*$-analytic for some $\rho^*\in \mathbb{R}_{>0}$. If, in addition $\mathcal{N}[u_n]$ and $\nabla\mathcal{N}[u_n]$ are $C^0$-polynomial-time computable functions for all $u_n\in \Pi_{n}(\Omega)$, then the solution $u^*$ is $H^k$-computable in polynomial-time.

Conversely, if the PDE operator fails to preserve the regularity of the data, that is, if the solution $u^*$ is not $\rho$-analytic for any $\rho\in \mathbb{R}_{>0}$ but only satisfies $u^*\in AC^{k-1}(\Omega)\cap H^k(\Omega)$ for some $k>\frac{d}{2}$, then $u^*$ exhibits a complexity blowup and is $H^k$-computable in at least super-polynomial-time.
\end{corollary}
\vspace{-1em}\begin{proof}
For the first part of the statement, we follow the proof of Theorem~\ref{thm:compt}, and apply Proposition~\ref{prop:conv_an} yielding
\begin{equation}
    \|\hat{u}_{\theta^j}-u^*\|_{H^k(\Omega)}^2\leq \mathcal{O}\left(\rho^{-n}+(1-\mu\sigma/L)^{j}\right).
\end{equation}
This implies that the error is of order $\mathcal{O}(2^{-N})$ in $\mathcal{O}(N\log(1-\mu\sigma/L))$ iterations with 
$\mathcal{O}\left(F_\mathcal{N}(N\log(2)/\log(\rho)+1)^d\right)$ operations. Here $F_\mathcal{N}:\mathbb{N}\rightarrow\mathbb{N}$ is a polynomial function proportional to the complexity of the operation $\theta^j-\delta\tau\nabla\mathcal{L}^n[\hat{u}_{\theta^j}]$, which has polynomial complexity due to the assumptions on the operator $\mathcal{N}$, yielding the claim.

For the second part of the statement, we apply Lemma~\ref{lemm:comp_initial_data}, which ensures that the Dirichlet boundary condition $g$ and the source term $f$ are computable in polynomial-time. Furthermore, by Theorem~\ref{thm:compt} the solution is $H^k$-computable. However, by Theorem~\ref{thm:blowup} together with the Sobolev embedding theorem \cite{brezis2011}, Chapter 9, for $k>d/2$, there exists a sub-exponential function $\mathcal{R}_\mathrm{sub}:\mathbb{N}\to\mathbb{R}+$ such that
\begin{equation}
    \|\hat{u}_{\theta^j}-u^*\|_{H^k(\Omega)}^2\geq\|\hat{u}_{\theta^j}-u^*\|_{C^0(\Omega)}^2 \geq \mathcal{R}_\mathrm{sub}(n).
\end{equation}
 Therefore, achieving an error of order 
 $$\|\hat{u}_{\theta^j}-u^*\|_{H^k(\Omega)}^2\leq 2^{-N},$$ requires at least $\mathcal{C}_\mathrm{sup}(N)$ operations, where $\mathcal{C}_\mathrm{sup}:\mathbb{N}\rightarrow \mathbb{N}$ is a super-polynomial function. In particular, this shows that the solution is not computable in polynomial-time.  
\end{proof}
\subsection{Scope of the Variational Framework}\label{subsec:Scope}
This subsection clarifies the scope and limitations of the computability and
complexity results presented in this work, with particular emphasis on the role
of the underlying computability notions and their implications for PDE
solutions.

\begin{remark}[Scope of the computability framework]
As shown in Chapter~1 of~\cite{Pour-El1989-yg}, Definition~\ref{def:C_0_comp} is
equivalent to the classical notion of Turing computability~\cite{Grzegorczyk1957},
namely sequential computability together with effective uniform continuity.
However, alternative notions of computability exist and do not necessarily
coincide with the one adopted in this document. In particular, the computability framework
introduced in Definition~\ref{def:comp_bsp} describes a proper subclass of the
classical Banach-Mazur computable functions~\cite{Boche2026-kq}; namely, 
every $X$-computable function is Banach-Mazur computable, while the converse does
not hold. Consequently, all computability results established in this work must
be interpreted relative to the specific framework defined in
Definition~\ref{def:comp_bsp}.
\end{remark}

\begin{remark}[Scope of complexity for PDE solutions]\label{rk:Complexity_PDEs}
The $H^{k}$-computability results for PDE solutions in
Theorem~\ref{thm:compt} are obtained via sequences of discrete Euler iterations
that converge effectively to the corresponding infinite-dimensional PDE solution.
These results are therefore intrinsically tied to the variational gradient-flow
framework developed in this paper and provide \emph{sufficient, but not
necessary}, conditions for computability and polynomial-time complexity, as a consequence of the convergence bounds in ~\eqref{Eq:erro_dec} and~\eqref{eq:conv_upper_bnd}. In particular, the non-existence of alternative approximation schemes that yield computable PDE solutions under different assumptions cannot be concluded from our results.

More generally, complexity characterizations of PDE solutions depend strongly on
the chosen notion of computability. Definition~\ref{def:pol_comp} introduces an
\emph{approximation-restricted} notion of polynomial-time computability, which
requires the existence of exponentially convergent sequences of polynomials that are computable in polynomial-time. This notion is strictly stronger than
Ko-Friedman polynomial-time computability~\cite{Steinberg2017}. Specifically, as shown in Theorem~\ref{thm:equival_compt}, $C^0$-polynomial-time computability in the sense of Definition~\ref{def:pol_comp} implies Ko–Friedman polynomial-time computability for Lipschitz continuous functions; however, the converse implication does not hold. A prototypical counterexample is $f(x):=|x|$, which is
polynomial-time computable by direct evaluation, yet admits no exponentially
convergent sequence of polynomials in the $C^0$-norm, as a consequence of Corollary~\ref{cor:onlyif_an}.
\end{remark}

\begin{remark}[Scope of complexity blow-up]
The complexity blow-up result in Corollary~\ref{cor:complexity} is intrinsically
linked to the notions of $X$-computability, introduced in
Definition~\ref{def:comp_bsp}, and polynomial-time complexity, as formalized in
Definition~\ref{def:pol_comp}. Once computability in $X(\Omega)$ is established,
the sub-exponential convergence lower bound implies that no polynomial
approximation scheme can converge exponentially fast. Consequently, any
admissible polynomial approximation must exhibit super-polynomial computational
complexity. This result is therefore invariant with respect to the particular
discretization scheme used to construct the discrete gradient-flow sequence of
polynomials, but not to the computability framework. 

Additional sources of computational hardness may induce complexity blow-up beyond
the loss of regularity of the PDE solution; such effects are not captured within the framework developed in this work. Examples include linear ODEs exhibiting
complexity blow-up~\cite{Boche2021} without any loss of regularity of the PDE solution. Conversely, there may exist PDE solutions
that exhibit complexity blow-up under the computability notion of
Definition~\ref{def:C_0_comp}, but not under alternative computability frameworks.
A deeper analysis of these distinctions is beyond the scope of the present work.
\end{remark}

In the next section, we present two examples illustrating how the theoretical framework developed above can be applied to a linear and a non-linear equation.
\section{Complexity of PDE Solutions: Examples}\label{sec:Experiments}
In this section, we present two examples illustrating the variational framework of Theorem~\ref{thm:compt}. These include a case where the solution is computable in polynomial-time, as well as a case where the solution exhibits a complexity blowup.
\subsection{Complexity of solutions to the Poisson Equation}
We start by defining the Poisson forward problem.
\begin{definition}\label{Def:Poisson_FP}
    Let $g\in H^{1/2}(\partial\Omega)$ be the Dirichlet boundary condition and $f\in L^2(\Omega)$ the source term. We define the \emph{Poisson forward loss} $\mathcal{L}:H^1(\Omega)\rightarrow\mathbb{R}_+$ as 
    \begin{equation}
        \mathcal{L}[u]:= \|\Delta u-f\|_{H^{-1}(\Omega)}^2+\|u|_{\partial\Omega}-g\|_{H^{1/2}(\partial\Omega)}^2.
    \end{equation}
\end{definition}
Next, we state the result, ensuring the well-posedness of the Poisson variational problem. A similar assertion can be found in Chapter 1 of~\cite{Bochev2009-ut}.
\begin{theorem}\label{thm:prop_poiss_loss}
    The Poisson loss $\mathcal{L}:H^1(\Omega)\rightarrow\mathbb{R}_+$ from Definition~\ref{Def:Poisson_FP} is a $\lambda$-convex, H\"older continuous functional,
    satisfying the apriori coercivity estimate 
    \begin{equation}\label{eq:apr_est}
        \|u\|_{H^1(\Omega)}^2\lesssim \mathcal{L}[u]+1.
    \end{equation}
    Therefore, there exists a unique solution $u^*\in H^1(\Omega)$ minimizing $\mathcal{L}$.
\end{theorem}
\vspace{-1em}\begin{proof}
    The proof of Theorem~\ref{thm:prop_poiss_loss} is included in Appendix~\ref{Appendix_C}.
\end{proof}
Next, we show how under certain assumptions the Poisson forward loss from Definition~\ref{Def:Poisson_FP} is computable in polynomial-time.
\begin{theorem}\label{thm:comp_poisson}
Let $\mathcal{L}: H^1(\Omega) \to \mathbb{R}+$ denote the Poisson forward loss introduced in \eqref{Def:Poisson_FP} in the hypercube $\Omega:=(-1,1)^d$. Suppose that the source term $f$ satisfies $f \in AC^{k-1}(\Omega) \cap BV^{k}(\Omega)$ and the Dirichlet boundary condition $g$ satisfies $g \in AC^{k-1/2}(\partial\Omega) \cap BV^{k+1/2}(\partial\Omega)$ for some $1<k \in \mathbb{N}$. Then the minimizer $u^* \in H^1(\Omega)$ of $\mathcal{L}$ is $H^1$-computable. 

Furthermore, if $f$ is $\rho$-analytic in $\Omega$ and $g$ is $\rho$-analytic in $\partial\Omega$, then the solution $u^*$ is $H^1$-computable in polynomial-time.
\end{theorem}
\begin{remark}
    Note that, in general, the assumptions in Theorem~\ref{thm:comp_poisson} are not sufficient to ensure that the solution $u^*$ inherits the regularity of the data, due to the presence of corners in the hypercube $[-1,1]^d$. To address this issue, we assume additional compatibility conditions \cite{Hell2014-wc} on the boundary term $g$. A detailed analysis of how the regularity of the domain influences the complexity of the solution will be the subject of future work.
\end{remark}
Before proving the theorem, we provide the following lemma.
\begin{lemma}\label{lemm:grad_SC}
Let $f\in L^2(\Omega)$ and $g\in H^{1/2}(\Omega)$ be polynomial-time computable functions and $\mathcal{L}^n:\mathcal{S}_n(\Omega) \rightarrow \mathbb{R}_+$ be a Sobolev cubature loss, defined as 
\begin{equation}\label{eq:sob_cubpoisson}
    \mathcal{L}^n[\hat{u}_\theta]:= \|\Delta \hat{u}_\theta-f_n\|_{L^2(\Omega)}^2 + n\| \hat{u}_\theta|_{\partial\Omega}-g_n\|_{L^2(\partial\Omega)}^2,\textrm{ for all } \hat{u}_\theta\in \Pi_n(\Omega),
\end{equation}
where $f_n:=\mathcal{I}_n[f]$ and $g_n:=\mathcal{I}_n[g]$. If the coefficients $\theta\in \Theta$ are computable in polynomial-time, then the $L^2$-gradient of the loss $\nabla\mathcal{L}^n[\hat{u}_\theta]\in \Pi_{n}(\Omega)$ is computable in polynomial-time for all $n\in \mathbb{N}$.
\end{lemma}
\begin{proof}
The proof of Lemma~\ref{lemm:grad_SC} is included in the appendix~\ref{Appendix_C}    
\end{proof}
Building on this result, we now proceed to prove the computability of the solution $u^*\in H^1(\Omega)$.
\vspace{-0em}\begin{proof}[Proof of Theorem~\ref{thm:comp_poisson}.]
We start by noting that Lemma~\ref{lemm:grad_SC} implies that the explicit Euler iterate 
$\theta^{k+1} = \theta^k-\delta \tau \nabla\mathcal{L}^n[\theta^{k}]$ 
can be computed in polynomial-time, assuming that the initial guess $\theta_0\in \Theta$ and the time-step $\delta\tau$ are computable in polynomial-time. 

Next, we show that the functional $\theta \mapsto \mathcal{L}^n[\hat{u}_\theta]$ is $\lambda$-convex. Using that 
$$
\|Q_n\|_{H^{1/2}(\partial\Omega)}^2\leq n\|Q_n\|_{L^2(\partial\Omega)}^2,\textrm{ for all } Q_n\in \Pi_{n}(\Omega),
$$
as shown in \cite{Diethelm2002-kr} and $\|u\|_{H^{-1}(\Omega)}^2\leq \|u\|_{L^2(\Omega)}^2$ by definition of the negative order norm, we obtain 
\begin{align*}
\|\Delta \hat{u}_\theta\|_{L^2(\Omega)}^2 +n \| \hat{u}_\theta|_{\partial\Omega}\|_{L^2(\partial\Omega)}^2\geq \|\Delta  \hat{u}_\theta\|_{H^{-1}(\Omega)}^2 +\| \hat{u}_\theta|_{\partial\Omega}\|_{H^{1/2}(\partial\Omega)}^2,
\end{align*}
for all $\hat{u}_\theta\in \Pi_n(\Omega)$. Therefore, by Theorem~\ref{thm:prop_poiss_loss}, we have that 
\begin{equation}
\begin{aligned}
        \langle\nabla\mathcal{L}^n_\theta[\theta_2]-\nabla\mathcal{L}^n_\theta[\theta_1],\theta_2-\theta_1 \rangle_{2} &= \|\Delta (\hat{u}_{\theta_2}-\hat{u}_{\theta_1})\|_{L^2(\Omega)}^2 +n \|\hat{u}_{\theta_2}|_{\partial\Omega}-\hat{u}_{\theta_1}|_{\partial\Omega}\|_{L^2(\partial\Omega)}^2\\&\geq \|\Delta (\hat{u}_{\theta_2}-\hat{u}_{\theta_1})\|_{H^{-1}(\Omega)}^2 +\|\hat{u}_{\theta_2}|_{\partial\Omega}-\hat{u}_{\theta_1}|_{\partial\Omega}\|_{H^{1/2}(\partial\Omega)}^2\\&
        \gtrsim \|\hat{u}_{\theta_2}-\hat{u}_{\theta_1}\|_{H^1(\Omega)}^2.
    \end{aligned}
\end{equation}
The proof of the $\lambda$-convexity follows by noting that the map $\theta\mapsto \|\hat{u}_
\theta\|_{H^1(\Omega)}^2$ is coercive for all $\hat{u}_\theta\in \Pi_{n}(\Omega)$. Indeed, let $\mathbb{T}_\Omega:=\left(T_i(P_n(\Omega))\right)_{i\in A_{n,d}}$, be the Vandermonde matrix of the Lagrange basis of degree $n\in \mathbb{N}$ in the Legendre grid of degree $n\in \mathbb{N}$. Then the matrix $\mathbb{T}_\Omega^T\mathbb{W}\mathbb{T}_\Omega\in \mathbb{R}^{|A_{n,d}|\times |A_{n,d}|}$ is positive definite, since the column vectors $T_i(P_{n}(\Omega))\in \mathbb{R}^{|A_{n,d}|}$ are linearly independent. Therefore 
\begin{equation}
     \|\hat{u}_
\theta\|_{H^1(\Omega)}^2\geq  \|\hat{u}_
\theta\|_{L^2(\Omega)}^2 \gtrsim \sigma\|\theta\|_1^2,
\end{equation}
where $\sigma\in \mathbb{R}_{>0}$ is the smallest eigenvalue of $\mathbb{T}_\Omega^T\mathbb{W}\mathbb{T}_\Omega\in \mathbb{R}^{|A_{n,d}|\times |A_{n,d}|}$. Consequently,

\begin{equation}
    \langle\nabla\mathcal{L}^n_\theta[\theta_2]-\nabla\mathcal{L}^n_\theta[\theta_1],\theta_2-\theta_1 \rangle_{L^2(\Omega)}\gtrsim \sigma \|\theta_2-\theta_1\|_1^2,
\end{equation}
as claimed. 
Next, we show that the gradient $\nabla\mathcal{L}^n:\Pi_n(\Omega)\rightarrow \Pi_n(\Omega)$ is a Lipschitz continuous map. Let $\hat{u}_{\theta_1},\hat{u}_{\theta_2}\in \Pi_{n}(\Omega)$, then by definition of the gradient and the Sobolev cubatures, we obtain 
\begin{equation}
\begin{aligned}
        \|\nabla\mathcal{L}^n_\theta[\theta_1]-\nabla\mathcal{L}^n_\theta[\theta_2]\|_1&= \biggr\|\int\limits_{\Omega}(\Delta^*\Delta \hat{u}_{\theta_1}\nabla_{\theta}\hat{u}_{\theta_1}-\Delta^*\Delta\hat{u}_{\theta_2}\nabla_{\theta}\hat{u}_{\theta_2})dx\\&\,\,\,\,\,\,\,\,\,+n\int\limits_{\partial\Omega}\hat{u}_{\theta_1}\nabla_{\theta}\hat{u}_{\theta_1}-\hat{u}_{\theta_1}\nabla_{\theta}\hat{u}_{\theta_2}ds\biggr\|_1
        \\&
       =  \|(\mathbb{T}_{\Omega}^T\mathbb{D}_\Delta^T\mathbb{W}\mathbb{D}_\Delta\mathbb{T}_{\Omega}+ n\mathbb{T}_{\partial\Omega}^T\mathbb{W}\mathbb{T}_{\partial\Omega})(\theta_2-\theta_1)\|_1\\&
       \leq L\|\theta_2-\theta_1\|_1,
\end{aligned}
\end{equation}
where $\mathbb{D}_\Delta:=\sum\limits_{i=1}^d\mathbb{D}_{x_i}^2$, and $L:= \|\mathbb{T}_{\Omega}\mathbb{D}_\Delta^T\mathbb{W}\mathbb{D}_\Delta\mathbb{T}_{\Omega}+n\mathbb{T}_{\partial\Omega}^T\mathbb{W}\mathbb{T}_{\partial\Omega}\|_\infty$.

Let $\theta^j\in \Theta$ be the $j$-th iterate of the explicit Euler gradient flow with time-step $\delta\tau\in \mathbb{R}_{>0}$. Since $\mathcal{L}^n$ is a $\lambda$-convex and $L$-Lipschitz smooth functional, then by Theorem~\ref{thm:erro_dec} we obtain
\begin{equation}
    \|\hat{u}_{\theta^j}-u^*\|_{H^1(\Omega)}^2\leq \epsilon_\mathrm{app}(n)+\epsilon_\mathrm{int}(n)+\epsilon_\mathrm{opt}(j),
\end{equation}
with $\epsilon_\mathrm{opt}(j) = (1-\sigma/L)^{j}$.

To finish the proof, we use that the Laplace operator is a hypo-elliptic operator \cite{Mitrea2013-cl}, hence, the solution $u^*\in H^1(\Omega)$ inherits the regularity of the domain, boundary condition and right hand-side. Therefore, as a result of Corollary~\ref{cor:complexity}, we can characterize the complexity of the solution as follows:
\begin{itemize}
    \item If $u^*\in AC^{k+1}(\Omega)\cap BV^{k+2}(\Omega)$, $f\in  AC^{k-1}(\Omega)\cap BV^{k}(\Omega)$ and $g\in  AC^{k-3/2}(\partial\Omega)\cap BV^{k-1/2}(\partial\Omega)$ for some $1<k\in \mathbb{N}$, then the solution is $H^1$-computable in at least super polynomial-time. 
    \item Moreover, if $f$ and $g$ are $\rho-$analytic functions, the solution $u^*\in H^1(\Omega)$ is $H^1$-computable in polynomial-time. 
\end{itemize}
\end{proof}
\subsection{Complexity of Solutions to the Eikonal Equation}
This subsection presents the Eikonal equation \cite{Alimov2019-np} as a representative example of a PDE whose solution is computable but exhibits a loss of regularity, resulting in a complexity blowup. We begin by formulating the corresponding variational problem.
\begin{definition}\label{Def:Eikonal_FP}
    Let $S\subseteq \Omega$ be a smooth manifold of co-dimension one, then the \emph{Eikonal forward loss} $\mathcal{L}:H^1(\Omega)\rightarrow\mathbb{R}_+$ is defined as
    \begin{equation}\label{Eq:Eikonal_FP}
        \mathcal{L}[u]:= \bigr\|\|\nabla u\|_1-1\bigr\|_{L^2(\Omega)}^2+\|u|_{S}\|_{H^{1/2}(S)}^2.
    \end{equation}
\end{definition}
Next, we present some of the analytical properties of the Eikonal loss. 
\begin{theorem}\label{thm:prop_eik_loss}
    The Eikonal loss $\mathcal{L}:H^1(\Omega)\rightarrow\mathbb{R}_+$ from \eqref{Eq:Eikonal_FP} is a continuous functional,
    satisfying the apriori coercivity estimate 
    \begin{equation}\label{eq:apr_est_eik}
        \|u\|_{H^1(\Omega)}\lesssim \mathcal{L}[u]+1.
    \end{equation}
    Moreover, the Signed Distance Function (SDF)~\cite{Alimov2019-np} $\phi\in H^1(\Omega)$ defined as
    \begin{equation}
        \phi(x):=\min\limits_{x_s\in S}||x-x_s||_1
    \end{equation}
    is a minimizer of the Eikonal loss. 
\end{theorem}
\begin{proof}
The proof of Theorem~\ref{thm:prop_eik_loss} is included in Appendix~\ref{Appendix_C}.   
\end{proof}
The following lemma provides a characterization of the regularity of one of the solutions of the Eikonal loss from Definition~\ref{Def:Eikonal_FP}.
\begin{lemma}\label{lemm:disc_sol}
    Let $S\subseteq \Omega:=(-1,1)^d$, $d\in \mathbb{N}$, be a smooth manifold of co-dimension one, and let $\phi_S:\Omega\rightarrow \mathbb{R}$ be the signed distance function of $S$, given by 
    \begin{equation}\label{Eq:SDF}
    \phi_S(x):=\min\limits_{x_s\in S}\|x-x_s\|_1.
    \end{equation}
    Furthermore, let $S^+:=\{\phi_S\geq 0\}$ and $M_S\subseteq S^+$ be the medial axis of $S$, defined as
    \begin{equation}
        M_S:=\{c\in S^+: \exists x_l,x_r\in S,\ x_l\neq x_r,\ \|c-x_l\|_1 = \|c-x_r\|_1\}.
    \end{equation}
    Then, the function $\phi_S$ is non-differentiable on $M_S$. More precisely, for every $c\in M_S$, there exist sequences $z_l^\delta,z_r^\delta\in S^+$ such that $z_l^\delta\to c$, $z_r^\delta\to c$, $\phi_S$ is differentiable at $z_l^\delta$ and $z_r^\delta$, and 
    \begin{equation}
        \nabla\phi_S(z_l^\delta)\to p_l \neq p_r \leftarrow \nabla\phi_S(z_r^\delta).
    \end{equation}
\end{lemma}
\vspace{-1em}\begin{proof}
 This is a well-known result in the signed distance function literature \cite{Alimov2019-np}, and we include a proof in Appendix~\ref{Appendix_C}. 
\end{proof}
\begin{theorem}\label{thm:compl_blowup}
    Let $S\subseteq \Omega$ be a closed smooth manifold of co-dimension one, and $\Psi:U\rightarrow \Omega$ a parametrization of the manifold, for some open set $U\subseteq \mathbb{R}^d$, and $\phi\in H^1(\Omega)$ be the SDF from ~\eqref{Eq:SDF}. If the parametrization $\psi$ satisfies additionally $\psi\in AC^{k-1}(\Omega)\cap BV^k(\Omega)$ for some $k\geq 1$, then  the SDF solution $\phi$ is $L^2$-computable. 
    
Moreover, for $\Omega:=(-1,1)$, if the parametrization 
$\Psi : U \rightarrow \Omega$ is $H^1$–computable in polynomial-time, then the 
corresponding SDF $\phi \in H^{1}(\Omega)$ solving the Eikonal equation necessarily exhibits a complexity blow-up.
\end{theorem}
\vspace{-1em}\begin{proof}[Proof of Theorem~\ref{thm:compl_blowup}.]
   Following Theorem~\ref{thm:prop_eik_loss}, we know that the SDF $\phi\in H^1(\Omega)$ is a minimizer of the Eikonal loss in~\eqref{Eq:Eikonal_FP}. Moreover, using the regularity assumptions $\Psi\in AC^{k-1}(\Omega)\cap BV^k(\Omega)$ for $k\geq 1$, and $\phi\in AC(\Omega)\cap BV^1(\Omega)$,  it follows from Lemma~\ref{lemm:comp_initial_data} that the SDF $\phi$ is $C^0$-computable, and consequently $L^2$-computable.

   To prove the complexity blowup of the solution, we assume that \( S \subseteq \Omega \) is given by a polynomial-time, \( H^1 \)-computable parametrization \( \Psi: U \rightarrow \Omega \), for some open set \( U \subseteq \mathbb{R}^d \). By Lemma~\ref{lemm:disc_sol}, there exists a subset \( M_s \subseteq \{\phi\geq 0\} \) (the medial axis of the manifold $S$) where the SDF \( \phi \) is not differentiable, and therefore the function $\phi$ is not \( \rho \)-analytic. Hence, combining the Sobolev embedding $H^1(\Omega)\hookrightarrow C^0(\Omega)$ and  Corollary~\ref{cor:complexity}, it follows that the solution \( \phi \) cannot be \( H^1 \)-computable in polynomial-time, establishing the claim. 

   Alternatively, if the parametrization $\Psi$, is $C^0$-computable in polynomial-time, then $\phi$ is $C^0$-computable in super-polynomial-time, yielding a complexity blowup in $C^0$ as well.
\end{proof}
\section{Conclusion}
In this paper, we introduced a variational framework for analyzing the computational complexity of partial differential equations (PDEs) without requiring closed-form solutions, thereby generalizing results from previous works~\cite{bacho2023complexityblowupsolutionslaplace,Pour-El1989-yg}. Specifically, the variational framework allowed us to estimate the complexity of a discrete gradient flow approximating the solution operator of the variational problem, and used the convergence rates of these discretizations to quantify the complexity of the infinite-dimensional PDE solutions. This approach enabled us to relate structural properties of PDE operators and their associated variational losses to the computational complexity of the corresponding solutions within the adopted computability and representation framework from Definition~\ref{def:comp_bsp}. Moreover, we introduced a method for proving complexity blowup in the PDE solution, depending on the regularity-preserving properties of the solution operators. Specifically, if the solution loses regularity relative to analytic boundary and/or initial data, then the solution exhibits a complexity blowup. This result links the regularity and approximation properties of PDE solutions to complexity theory via effective polynomial representations.

In particular, Theorem~\ref{thm:compt} establishes $H^k$-computability for the solutions of a class of (possibly nonlinear) PDEs, under the assumptions that the associated variational loss is coercive, convex, satisfies the quadratic growth condition, and is Lipschitz smooth. Furthermore, Corollary~\ref{cor:complexity} provides a quantitative characterization of the solution complexity within this framework, identifying sufficient conditions under which the solution is computable in polynomial-time, and conditions under which it exhibits complexity blowup. We illustrate these results with two examples: (i) the Poisson equation, whose solution—assuming analytic boundary data—is $H^1$-computable in polynomial-time; and (ii) the Eikonal equation, which, despite being $L^2$-computable under analytic boundary conditions, exhibits complexity blowup due to loss of regularity in the solution.

To summarize, we have developed a framework that builds a rigorous bridge between computability, complexity theory, and variational formulations of PDEs. By providing upper and lower bounds on solution complexity within the chosen computability and discretization setting. This framework, delineates regimes in which PDEs admit effective numerical approximation schemes and those in which complexity blowup occurs. Moreover, it establishes a direct connection between structural properties of PDE forward losses—such as convexity and coercivity—and the regularity of the solution operator with the resulting complexity of the infinite-dimensional solution. Thus, the framework strengthens the theoretical foundations of PDE computability and complexity in this representation-dependent context and provides constructive insights for the design of efficient numerical schemes.

The main limitation of our framework is that the results depend on the chosen effective polynomial representation of computable functions and on the discrete gradient-flow iterates used to estimate the complexity of the PDE solution. Consequently, the statements derived in this work provide sufficient conditions for computability and complexity within a representation- and discretization-dependent setting, rather than universal, model-independent complexity-theoretic barriers for PDEs. Moreover, the error decomposition in Theorem~\ref{thm:erro_dec} yields only upper bounds on the solution’s complexity. Within the present framework, complexity blowup can therefore be established only by identifying a loss of regularity of the solution $u^* \in H^k(\Omega)$, and for $k>\frac{d}{2}$

Further research is required to strengthen the lower-bound analysis, potentially extending it to higher dimensions and incorporating lower bounds on the convergence behavior of the gradient flow itself. Such developments could provide alternative characterizations of complexity blowup beyond the current representation-dependent framework. In addition, extending the approach toward algorithm-independent complexity lower-bounds—by exploiting the intrinsic continuous structure of the gradient flow—may lead to stronger and more general discretization-independent complexity characterizations of PDE solutions applicable to broader classes of computational methods.

\appendix
\section{Appendix}\label{appendix_A}
In this section, we present the characterization of the computability structure in a Banach space, the proof of Theorem~\ref{thm:compt_structures} and of Proposition~\ref{prop:comp_cub}. 

We start with the definition of Banach spaces with computable structures.
\begin{definition}\label{def:comp_str}
A Banach space $X$ is said to have a \emph{computability structure}, if the following axioms are fulfilled:
\begin{enumerate}
    \item Let $(x_n)_{n\in \mathbb{N}},(y_n)_{n\in \mathbb{N}}\subseteq X$ be computable sequences in $X$, $(\alpha_{n,k})_{n,k\in \mathbb{N}}, (\beta_{n,k})_{n,k\in \mathbb{N}}\subseteq\mathbb{R}$ be computable double sequences in $\mathbb{R}$, and $d:\mathbb{N}\rightarrow \mathbb{N}$ be a recursive function. Then, the sequence
    \begin{equation}
        S_n:=\sum\limits_{k=1}^{d(n)}(\alpha_{n,k}x_n+\beta_{n,k}y_n),
    \end{equation}
    is computable in $X$.
    \item Let $(x_{n,k})_{n,k\in \mathbb{N}}\subseteq X$ be a computable double sequence in $X$ such that $x_{n,k}\stackrel{k\rightarrow\infty}{\longrightarrow}x_n$ effectively in $k$ and $n$. Then $(x_n)_{n\in \mathbb{N}}$ is computable in $X$. 
    \item If $(x_n)_{n\in \mathbb{N}}$ is a computable sequence in $X$, then $(\|x_n\|_X)_{n\in \mathbb{N}}$ is a computable sequence. 
\end{enumerate}
\end{definition}
Next, we show Theorem~\ref{thm:compt_structures}.
\vspace{-0.0em}\begin{proof}[Proof of Theorem~\ref{thm:compt_structures}.]
    We start by proving the first axiom. Let $(x_n)_{n\in \mathbb{N}},(y_n)_{n\in \mathbb{N}}\subseteq X$, be computable sequences in $X$, $(\alpha_{n,k})_{n,k\in \mathbb{N}}, (\beta_{n,k})_{n,k\in \mathbb{N}}\subseteq\mathbb{R}$ be computable double sequences, and $d:\mathbb{N}\rightarrow \mathbb{N}$ be a recursive function. Then, using the density assumption for $X$, there exist $C^0$-computable double sequences $(f_{n,k})_{n,k\in \mathbb{N}},(q_{n,k})_{n,k\in \mathbb{N}}$ converging effectively to $(x_n)_{n\in \mathbb{N}},(y_n)_{n\in \mathbb{N}}$, i.e., there exist recursive functions $e_x:\mathbb{N}\times\mathbb{N}\rightarrow\mathbb{N}$, $e_y:\mathbb{N}\times \mathbb{N}\rightarrow\mathbb{N}$ such that
    \begin{align}
        &\|f_{n,k}-x_n\|_{X}\leq 2^{-N}\\&
        \|q_{n,k}-y_n\|_{X}\leq 2^{-N},
    \end{align}
    for all $n\in \mathbb{N}$, with $k\geq e_x(n,N)$ for the first estimate and $k\geq e_y(n,N)$ for the second. Due to the vector structure of the Banach spaces we know that $\alpha_{n,k}x_n + \beta_{n,k}y_n\in X$. Moreover, the double sequence 
    \begin{equation}
        \tilde{S}_{n,k}:=\sum \limits_{l=1}^{d(n)}(\alpha_{n,l}f_{n,k}+\beta_{n,l}q_{n,k}),
    \end{equation}
    is $C^0$-computable, as it is a computable amount of sums and products of computable functions. To conclude, we use the triangle inequality property of the norm in $X$, and the previous observations, yielding for all $k\geq e_x(n,N)$,
    \begin{align*}
    \|S_n-\tilde{S}_{n,k}\|_{X}&\leq \sum\limits_{l=1}^{d(n)}\alpha_{n,l}\|x_n-f_{n,k}\|_X+\beta_{n,k}\|y_n-q_{n,k}\|_X\\&
    \leq C_{\alpha,\beta}(n) 2^{-N},
    \end{align*}
    with $C_{\alpha,\beta}(n):=2d(n)\max\limits_{l\in\{1,\dots,d(n)\} }\{\alpha_{n,l},\beta_{n,l}\}$, yielding the claim. 

    For the second axiom, let $(f_{n,k,l})_{n,k,l\in \mathbb{N}}$ be a $C^0$-computable triple sequence converging effectively to $(x_{n,k})_{n,k\in \mathbb{N}}$ as $l\rightarrow \infty$, by definition of $X$-computability. Then applying the triangle inequality, the following holds
    \begin{equation}
    \|f_{n,k.l}-x_{n}\|_X\leq \|f_{n,k,l}-x_{n,k}\|_X + \|x_{n,k}-x_n\|_X.
    \end{equation}
    Hence, taking the limit
    $l\rightarrow\infty$ we get that there exists a recursive $e:\mathbb{N}\times \mathbb{N}\rightarrow\mathbb{N}$ such that for all $k\geq e(n,N)$, the following holds
    \begin{equation}
    \|f_{n,k}-x_n\|_X\lesssim 2^{-N},
    \end{equation}
    for all $n\in \mathbb{N}$, where $f_{n,k}:=\lim\limits_{l\rightarrow \infty}f_{n,k,l}$. Therefore, the sequence $(x_n)_{n\in \mathbb{N}}\subseteq X$ is $X$-computable. 
    
    To prove the last axiom, we use that there exists a $C^0$-computable double sequence $(q_{n,j})_{n,j\in \mathbb{N}}$ and a recursive function $e:\mathbb{N}\times \mathbb{N}\rightarrow\mathbb{N}$ such that 
    \begin{equation}
        \|q_{n,j}-x_n\|_X\leq 2^{-N},
    \end{equation}
    for all $j\geq e(n,N)$. Hence, applying the reverse triangle inequality we have that 
    \begin{equation}
        |\|q_{n,j}\|_X-\|x_n\|_{X}|\leq \|q_{n,j}-x_n\|_X\leq 2^{-N}.
    \end{equation}
    Consequently, we obtained a $C^0$-computable double sequence $f_{n,k}:=\|q_{n,j}\|_X$ converging effectively to $\|x_n\|_X$. 
\end{proof}
We follow to proof Proposition~\ref{prop:comp_cub}, but before we introduce the following lemma. 
\begin{lemma}\label{comp_Leg_grid}
    The Legendre grid $P_n(\Omega)\subseteq \Omega$ given as the roots of the Legendre polynomials can be computed in polynomial-time for all $n\in \mathbb{N}$. 
\end{lemma}
\vspace{-1em}\begin{proof}
One way to prove the polynomial computability of the Legendre grid and weights is to use the Golub–Welsch (GW) algorithm \cite{GW-Method}. For simplicity, consider the case $d=1$. Let $J_n \in \mathbb{R}^{(n+1)\times (n+1)}$ be the Jacobi matrix, defined as the tri-diagonal matrix with entries $(J_n)_{i,j} = 0$ for $j\not\in \{i+1,i\}$ for $i\in \{1,\dots,n\}$ and $(J_n)_{i-1,i} = \beta_{i+1}$,$(J_n)_{i+1,i} = \beta_{i}$ for $i\in \{2,\dots,n\}$, 
where $\beta_i:= \frac{i}{\sqrt{(2i-1)(2i+1)}}$. The Legendre nodes $P_n(\Omega)$ are then obtained as the eigenvalues of $J_n$. Since the coefficients $\beta_i$ are computable in polynomial-time, and the tri-diagonal structure allows eigenvalues of $J_n$ to be computed in $\mathcal{O}(n^2)$ operations \cite{Hale2013-zy}, this yields the desired complexity bound.

An alternative proof of polynomial computability uses more modern approaches, such as Newton–Raphson iterations with carefully chosen initializations \cite{Petras1999-kp}. These methods achieve improved complexity, requiring only $\mathcal{O}(n)$ \cite{Hale2013-zy} or even $\mathcal{O}(1)$ \cite{Bogaert2012-wf} operations, and currently represent the state of the art for computing Legendre quadratures.
\end{proof}
\vspace{-1em}\begin{proof}[Proof of Proposition~\ref{prop:comp_cub}]
We start by showing that the interpolation $f_n:=\mathcal{I}_n[f] \in \Pi_n(\Omega)$ in Lagrange basis from Remark~\ref{rk:interpolation} is an $H^k$-computable function.

First, observe that the Legendre polynomials can be expressed via the Rodriguez formula \cite{Arfken2013-ny} as derivatives of the polynomial $(x^2 - 1)^n$, i.e., the 1-D Legendre polynomial of degree $n\in \mathbb{N}$ is given by
$$
P_n(x) = \frac{1}{2^n n!} \frac{d^n}{dx^n} (x^2 - 1)^n.
$$
Since $(x^2 - 1)^n$ is analytic and $C^0$-computable, its derivatives are also $C^0$-computable as presented in Chapter 1 of~\cite{Pour-El1989-yg} and therefore $H^k$-computable. In higher dimensions, the Legendre polynomial $P_\alpha\in\Pi_n(\Omega)$ for $\alpha\in A_{n,d}$ is given as a finite product of 1-D basis elements. Consequently, the Legendre polynomials themselves are $C^0$-computable and therefore $H^k$-computable.

Moreover, the roots of these polynomials are computable \cite{Pour-El1989-yg}, Chapter 1, ensuring that the Legendre grid $P_{n}(\Omega) \subseteq \Omega$ is computable. By the $H^k$-computability of $f$, the finite sequence $f(P_{n}(\Omega)):=(f(p_\alpha))_{\alpha\in A_{n,d}} \subseteq \mathbb{R}$ is also computable.

The $H^k$ computability of the Lagrange interpolation then follows from the $C^0$-computability of the Lagrange basis polynomials, along with the fact that finite sums and products of computable reals remain computable. Moreover, since the interpolant $f_n \in H^k(\Omega)$ is analytic and $H^k$-computable, its derivatives and their integrals are computable. Hence the Sobolev norm
\begin{equation}
    \|f_n\|_{H^k(\Omega)}^2 = \sum\limits_{|\beta|\leq k}\int\limits_{\Omega}(\partial^\beta f_n)^2dx,
\end{equation}
is computable.

The $H^k$-computability in polynomial-time follows by applying Lemma~\ref{comp_Leg_grid}, implying that $f(P_{n}(\Omega))$ is computable in polynomial-time, if $f$ is  $H^k$-computable in polynomial-time, and that $\partial^\beta L_\alpha\in \Pi_{n}(\Omega)$ is $H^k$-computable in polynomial-time for all $\alpha\in A_{n,d}$.
\end{proof}
Next we show the relation between the polynomial-time computability from Definition~\ref{def:complexity_blowup}, against the polynomial-time computability in \cite{Steinberg2017}. We start by recalling Ko-Friedmann's theorem \cite{Steinberg2017}. 
\begin{theorem}[\cite{Steinberg2017}]\label{thm:comp_L2}
Let $\mathbb{D}$ be the set of dyadic numbers and $\omega$ the set of integers in unary. A function $f \colon [0,1] \to \mathbb{R}$ is polynomial-time computable if and only if both of the following conditions are satisfied:
\begin{itemize}
    \item There exists a polynomial-time computable function
    \[
        \varphi \colon \mathbb{D} \times \omega \to \mathbb{D}
    \]
    such that for any $x \in [0,1] \cap \mathbb{D}$,
    \[
        \lvert \varphi(x,1^n) - f(x) \rvert < 2^{-n},
    \]
    where $\mathbb{D}\subseteq \mathbb{R}$ are the Dyadic numbers. 
    \item The function $f$ admits a polynomial modulus of continuity, i.e., for all $x,y \in [0,1]\cap \mathbb{D}$ and all $n \in \omega$,
\[
|x-y|\le 2^{-\mu(n)} \;\Longrightarrow\; |f(x)-f(y)| < 2^{-n},
\]
and if $\mu(n) \neq 0$, then $\mu(n+1) > \mu(n)$, that is, $\mu$ is strictly
increasing whenever it is nonzero.
\end{itemize}
\end{theorem}
Following this characterization, we provide the following result.
\begin{theorem}\label{thm:equival_compt}
If $f:[0,1]\to\mathbb{R}$ is a Lipschitz function that is $C^0$-polynomial-time computable in the sense of Corollary~\ref{cor:complexity}, then $f$ is polynomial-time computable according to the Ko-Friedman characterization given in Definition~\ref{thm:comp_L2}.
\end{theorem}
\begin{proof}
Let $f:[0,1]\rightarrow\mathbb{R}$ be a Lipschitz continuous function, which is $C^0$-polynomial-time computable. Then, there exists a sequence $(q_n)_{n\in \mathbb{N}}\subseteq \Pi_\infty([0,1])$ of polynomial-time computable polynomials converging exponentially to $f$, i.e., 
    \begin{equation}
        ||q_n-f||_{C^0([0,1])}\leq \mathcal{R}_{\textrm{exp}}(n),
    \end{equation}
    for all $n\in \mathbb{N}$. This implies that 
    \begin{equation}
        |q_n(x)-f(x)|\leq \mathcal{R}_{\textrm{exp}}(n),
    \end{equation}
    for all $x\in [0,1]\cap \mathbb{D}$. Consequently, for $\phi(x,1^n):=q_n(x)$, the following holds 
    \begin{equation}
        |\phi(x,1^n)-f(x)|\leq 2^{-N}
    \end{equation}
    for all $x\in [0,1]\cap\mathbb{D}$ in $\mathcal{O}(\mathcal{C}_\textrm{pol}(N))$ operations, establishing the first condition. 
    
    Moreover, by Lipschitz continuity of the function, the following holds
    \begin{equation}
        |f(x)-f(y)|\leq L|x-y|.
    \end{equation}
    Therefore, if $|x-y|\leq 2^{-\mu(N)}$, with $\mu(N):=N+\lceil\log L\rceil$, it holds
    \begin{equation}
        |f(x)-f(y)|\leq L|x-y|\leq L2^{-\mu(N)} = 2^{-
        N},
    \end{equation}
    yielding the polynomial modulus of continuity for $f$, and thus proving the claim.  
\end{proof}
The following remark discusses the other direction of the implication presented in Theorem~\ref{thm:equival_compt}.
\begin{remark}\label{remark_iff}
If we relax the notion of $C^0$-polynomial-time computability from
Definition~\ref{def:pol_comp} by allowing a sequence of polynomial-time computable functions that are not necessarily polynomials, then every Ko–Friedman polynomial-time computable function is also $C^0$-polynomial-time computable.

The drawback of this relaxation is that the complexity blow-up characterization
of Remark~\ref{rk:complexity_blowup} no longer applies, since slow polynomial
approximability does not preclude polynomial-time computability via alternative
(non-polynomial) approximation schemes.
\end{remark}
\begin{proof}
We show the first part of Remark~\ref{remark_iff}.
     Let $f$ be a polynomial-time computable function as in Theorem~\ref{thm:comp_L2}, then there exists a polynomial-time computable function $\phi:[0,1]\cap\mathbb{D}$ such that 
    \begin{equation}\label{eq:ineq_pf_comp}
        |\phi(x,1^n)-f(x)|<2^{-n},
    \end{equation}
    for all $x\in [0,1]\cap\mathbb{D}$. Since  $[0,1]\cap\mathbb{D}$ is a bounded domain, and $\phi$ is a computable function, there exists a sequence $(q_{m,n})_{m\in \mathbb{N}}\subseteq \Pi_\infty([0,1]\cap\mathbb{D})$ such that 
    \eqref{eq:ineq_pf_comp} is equivalent to 
    \begin{equation}\label{eq:ineq_pf_comp_2}
        ||q_{m,n}-\phi(\cdot,1^n)||_{C^0([0,1]\cap{\mathbb{D}})}<2^{-m},
    \end{equation}
    for all $n\in \mathbb{N}$. By triangle inequality, we have that
    \begin{align*}
        ||q_{m,n}-f||_{C^0([0,1]\cap\mathbb{D})}&\leq ||q_{m,n}-\phi(\cdot,1^n)||_{C^0([0,1]\cap\mathbb{D})}+||\phi(\cdot,1^n)-f||_{C^0([0,1]\cap\mathbb{D})}\\&
        <2^{-m}+2^{-n}.
    \end{align*}
    Consequently, by continuity of the norm, the following holds
    \begin{equation}
        ||q_{n}-f||_{C^0([0,1]\cap\mathbb{D})}<2^{-n},
    \end{equation}
    where $(q_n)_{n\in \mathbb{N}}:=(q_{n,n})_{n\in \mathbb{N}}$ is the diagonal sub-sequence. Finally, since $\phi$ is computable in polynomial-time, every $q_n$ is computable in polynomial-time, yielding the claim.
\end{proof}
\section{Appendix}\label{appendix_B}
In this section, we include the proofs of Lemmas~\ref{lemm:lam_convlin}, and~\ref{lemm:comp_initial_data}. 
\vspace{-0em}\begin{proof}{Proof of Lemma~\ref{lemm:lam_convlin}}
      Let $A^*:\mathrm{dom}(A^*)\subseteq H(\Omega)\rightarrow H(\Omega)$ be the formal adjoint of $A$. Then the gradient of the loss, with respect to the $H$-inner product, is given by
    \begin{equation}
        \nabla\mathcal{L}[u] = 2A^*(Au-b).
    \end{equation}
    Hence, the following holds
    \begin{equation}\label{eq:pf_lin_sc}
            \begin{aligned}
        \langle \nabla\mathcal{L}[u]-\nabla\mathcal{L}[v],u-v \rangle_{H(\Omega)} &= 2\langle A^*A(u-v),u-v \rangle_{H(\Omega)}\\&
        = 2\| A(u-v)\|_{H(\Omega)}^2\geq 2\mu \|u-v\|_{H(\Omega)}^2,
    \end{aligned}
    \end{equation}
    where the second equality follows from the definition of the adjoint operator $A^*$ , and the final inequality uses the coercivity assumption. We conclude by noting that the inequality in Eq. \ref{eq:pf_lin_sc} provides an equivalent characterization of $2\mu$-convexity compared to the one given in Definition \ref{def:lam_conv}.
\end{proof}
Next, we present the proof of Lemma~\ref{lemm:comp_initial_data}.
\vspace{0em}\begin{proof}{Proof of Lemma~\ref{lemm:comp_initial_data}}
    To prove that $h$ is computable in polynomial-time, we consider the Sobolev cubature approximation of the reconstruction loss from Definition \ref{Def:rec_prob} given by
    \begin{equation}
        \mathcal{L}^n_r[\hat{u}_\theta]:= \|\hat{u}_\theta-\mathcal{I}_n[h]\|_{H^s(\mathcal{B})}^2 = \sum\limits_{|\beta|\leq s}\mathfrak{r}^T(\mathbb{D}^{\beta})^T\mathbb{W}\mathbb{D}^{\beta}\mathfrak{r} = \mathfrak{r}^T\mathbb{W}_s\mathfrak{r},
    \end{equation}
    with $\hat{u}_\theta\in \Pi_{n}(\mathcal{B})$ and $\mathfrak{r}:=(\hat{u}_\theta-h)(P_n(\mathcal{B}))$. We note that the map $\theta\mapsto\|\hat{u}_\theta\|_{H^s(\mathcal{B})}^2$ is $\sigma$-coercive where $\sigma>0$ is the smallest eigenvalue of the full-column rank matrix $\mathbb{V}^T\mathbb{W}_s\mathbb{V}\in \mathbb{R}^{|A_{n,d}|\times|A_{n,d}|}$, with $\mathbb{V}:=(T_i(P_n(\mathcal{B})))_{i\in A_{n,d}}$, the Vandermonde matrix over the Legendre grid.  Indeed, the following holds
    \begin{equation}
        \|\hat{u}_\theta\|_{H^s(\mathcal{B})}^2 = \hat{\mathfrak{u}}_\theta^T\mathbb{W}_s\hat{\mathfrak{u}}_\theta = \theta^T\mathbb{V}^T\mathbb{W}_s\mathbb{V}\theta\geq \sigma\|\theta\|_1^2.
    \end{equation}
Therefore, by Lemma~\ref{lemm:lam_convlin}, $\theta\mapsto\mathcal{L}^n_r[\hat{u}_\theta]$ is a $2\sigma$-convex functional. Moreover, by definition of the gradient $\nabla\mathcal{L}^n$, the following holds
\begin{equation}
    \|\nabla\mathcal{L}^n[\hat{u}_{\theta_1}]-\nabla\mathcal{L}^n[\hat{u}_{\theta_2}]\|_1 = 2 \|\mathbb{V}^T\mathbb{W}_s\mathbb{V}(\theta_2-\theta_1)\|_1\leq L_r\|\theta_2-\theta_1\|_1,
\end{equation}
with $L_r:= 2\|\mathbb{V}^T\mathbb{W}_s\mathbb{V}\|_{\infty}$. Hence, the map $\theta\mapsto \nabla\mathcal{L}^n[\hat{u}_\theta]$ is $L_r$-Lipschitz continuous. 
By Remark~\ref{rk:lam_conv} and Theorem~\ref{thm:erro_dec}, the $j$-th iterate of the explicit Euler gradient flow $\hat{u}_{\theta^j}\in \Pi_{n}(\Omega)$ with timestep $\delta\tau \in \mathbb{R}_+$, converges with the rate
\begin{equation}
    \|\hat{u}_{\theta^j}-h\|_{H^k(\Omega)}^2\leq \epsilon_{\mathrm{app}}(n)+ \epsilon_{\mathrm{int}}(n)+\mathcal{O}\left((1-\sigma/L_r)^{- j}\right).
\end{equation}
Furthermore, applying Propositions \ref{prop:conv_an}, and \ref{prop:abs_cont} for the approximation error, and \ref{thm:sob_cub_s} for the integration error, we obtain that $\epsilon_{\mathrm{app}}(n)+ \epsilon_{\mathrm{int}}(n) = \mathcal{O}(n^{2s-k})$ if $h\in AC^{k-1}(\mathcal{B})\cap H^k(\mathcal{B})$ or $\epsilon_{\mathrm{app}}(n)+ \epsilon_{\mathrm{int}}(n) = \mathcal{O}(\rho^{-n})$ if $h$ is $\rho$-analytic. This implies the existence of recursive functions 
\(e_n:\mathbb{N}\to\mathbb{N}\) and \(e_j:\mathbb{N}\to\mathbb{N}\) such that  
\begin{equation}
    \|\hat{u}_{\theta^j}-h\|_{H^k(\mathcal{B})}^2 \leq 2^{-N},
\end{equation}
for all \(n \geq e_n(N)\) and \(j \geq e_j(N)\). Moreover, the asymptotic 
behavior of these functions satisfies
\[
    e_j(N) = \mathcal{O}\!\left(\frac{N\log 2}{\log(\sigma/L_r)}\right), \qquad
    e_n(N) =
    \begin{cases}
        \mathcal{O}\!\left(2^{N/(k-2s)}\right), & \text{if } h \in AC^{k-1}(\mathcal{B}) \cap H^k(\mathcal{B}), \\[6pt]
        \mathcal{O}\!\left(\tfrac{N\log 2}{\log(\rho)}\right), & \text{if } h \text{ is }\rho\text{-analytic}
    \end{cases}.
\]
In the case that $h$ is $\rho$-analytic, we can achieve a precision of $2^{-N}$ in  $\mathcal{O}(j(n+1)^d)=\mathcal{O}\left(N(N+\log(\rho)/\log(2))^d\right)$ operations. This is polynomial in $N$, thereby proving the claim.  
\end{proof}
\section{Appendix}\label{Appendix_C}
In this section, we include the proofs of Theorem~\ref{thm:prop_poiss_loss}, Lemma~\ref{lemm:grad_SC} and Lemma~\ref{lemm:disc_sol}. 

\vspace{-0em}\begin{proof}[Proof of Theorem~\ref{thm:prop_poiss_loss}]
We start by noting that due to the linearity of the differential operator, the loss $\mathcal{L}$ is convex. Next, we show the weak lower semi-continuity of the functional. 
Let $u,v\in H^1(\Omega)$ and $\mathcal{L}_\partial,\mathcal{L}_b:H^1(\Omega)\rightarrow \mathbb{R}_+$ denote the PDE and boundary condition loss terms defined such that $\mathcal{L}=\mathcal{L}_\partial+\mathcal{L}_b$. Then, by the reverse triangle inequality, the definition of the $H^{-1}$-norm, and the Cauchy-Schwarz inequality, the following holds
\begin{equation}\label{eq:H-1_ineq}
\begin{aligned}
    \left|\mathcal{L}_\partial[u]-\mathcal{L}_\partial[v]\right|&\leq \|\Delta (u-v)\|_{H^{-1}(\Omega)}^2
    = \sup\limits_{w\in H^1_0(\Omega)}\frac{\langle \Delta (u-v), w\rangle_{L^2(\Omega)}^2}{\|w\|_{H^1(\Omega)}^2}\\&
    = \sup\limits_{w\in H^1_0(\Omega)}\frac{\langle \nabla (u-v), \nabla w\rangle_{L^2(\Omega)}^2}{\|w\|_{H^1(\Omega)}^2}\leq \|\nabla(u-v)\|_{L^2(\Omega)}^2\\&
    \leq \|u-v\|_{H^1(\Omega)}^2.
\end{aligned}
\end{equation}
Similarly for the boundary term, using the reverse triangle inequality and the trace theorem, it holds
\begin{align*}
    \left|\mathcal{L}_b[u]-\mathcal{L}_b[v]\right|&\leq \|(u-v)|_{\partial\Omega}\|_{H^{1/2}(\partial\Omega)}^2\leq \|u-v\|_{H^1(\Omega)}^2.
\end{align*}
This implies that $\mathcal{L}$ is $2$-H\"older continuous. Consequently, the loss $\mathcal{L}$ is a convex and strongly continuous functional, and therefore weakly lower semi-continuous. 

Next, we show the coercivity estimate from \eqref{eq:apr_est}. Let $u\in H^1(\Omega)$ and $\mathcal{E}:H^{1/2}(\partial\Omega)\rightarrow H^1(\Omega)$ be the extension operator \cite{Necas2010-yx}. We define $u_b:=\mathcal{E}[u|_{\partial \Omega}]$ and $u_0:=u-u_b$. Then, by Cauchy-Schwarz, the triangle inequality and the extension theorem \cite{HAJLASZ1997221}, the following holds
\begin{equation}\label{eq:est_l2-1}
\begin{aligned}
    \|u\|_{L^2(\Omega)}^2&\leq   \|u_0\|_{L^2(\Omega)}^2+\| u_b\|_{L^2(\Omega)}^2\lesssim \|\nabla u\|_{L^2(\Omega)}^2+\|u_b\|_{H^1(\Omega)}^2\\&
    \lesssim \|\nabla u\|_{L^2(\Omega)}^2+\|u|_{\partial\Omega}\|_{H^{1/2}(\partial\Omega)}^2.
\end{aligned}
\end{equation}
Moreover, we claim that the following holds for all $u_0\in H^1_0(\Omega)$
\begin{equation}
    ||\nabla u_0||_{L^2(\Omega)}\stackrel{!}{\leq} \|\Delta u_0\|_{H^{-1}(\Omega)}.
\end{equation}
Indeed, we have that
\begin{align}
\|\nabla u_0\|_{L^2(\Omega)}^2 &= \langle \nabla u_0,\nabla u_0\rangle_{L^2(\Omega)} =||\nabla u_0||_{L^2(\Omega)}\langle \nabla u_0,\frac{\nabla u_0}{||\nabla u_0||_{L^2(\Omega)}}\rangle_{L^2(\Omega)} \\&
\leq||\nabla u_0||_{L^2(\Omega)}\|\Delta u_0\|_{H^{-1}(\Omega)},
\end{align}
dividing at both sides by $||\nabla u_0||_{L^2(\Omega)}$ yields the claim. 
Here we assumed, without loss of generality, that \(\|\nabla u_0\|_{L^2(\Omega)}^2 > 0\), since the inequality is trivial if \(\|\nabla u_0\|_{L^2(\Omega)}^2 = 0\).
Consequently, we have that
\begin{equation}\label{eq:est_nab-1}
\begin{aligned}
    \|\nabla u\|_{L^2(\Omega)}^2&\leq \|\nabla u_0\|_{L^2(\Omega)}^2+\|\nabla u_b\|_{L^2(\Omega)}^2\\&\leq \|\Delta u\|_{H^{-1}(\Omega)}^2+\|\Delta u_b\|_{H^{-1}(\Omega)}^2+\|\nabla u_b\|_{L^2(\Omega)}^2\\&
    \lesssim \|\Delta u\|_{H^{-1}(\Omega)}^2+\|\nabla u_b\|_{L^2(\Omega)}^2\lesssim  \|\Delta u\|_{H^{-1}(\Omega)}^2+\|u|_{\partial\Omega}\|_{H^{1/2}(\partial\Omega)}^2,
\end{aligned}
\end{equation}
where we used the extension theorem for the last inequality, and the fact that $\|\Delta u_b\|_{H^{-1}(\Omega)}^2\leq \| \nabla u_b\|_{L^2(\Omega)}^2$ following \eqref{eq:H-1_ineq}. Combining the estimates for $\|u\|_{L^2(\Omega)}^2$ and $\|\nabla u\|_{L^2(\Omega)}^2$ yields the estimate in \eqref{eq:apr_est}.

To finish the proof, we show the $\lambda$-convexity of the loss $\mathcal{L}$. Let $j^k:H^k(\Omega)\rightarrow L^2(\Omega)$ be the canonical embedding from $H^k$ to $L^2$ and $j^{k*}:\mathrm{dom}(j^{k^*})\subseteq L^2(\Omega)\rightarrow H^k(\Omega)$ its formal adjoint. Then we have that 
\begin{equation}
    \|u\|_{H^{-k}(\Omega)}^2 = \langle u, j^{k^*} u\rangle_{L^2(\Omega)},
\end{equation}
for all $u\in L^2(\Omega)$. Hence, the variations $d\mathcal{L}_\partial[u]:H^1(\Omega)\rightarrow\mathbb{R}$ of the PDE term at $u\in H^1(\Omega)$, in direction $\delta u\in H^1(\Omega)$, are given by
\begin{equation}
    d\mathcal{L}_\partial[u](\delta u) =  \langle \Delta \delta u, j^{k^*} (\Delta u -f)\rangle_{L^2(\Omega)}+\langle \Delta u -f, j^{k^*} \Delta \delta u \rangle_{L^2(\Omega)}.
\end{equation}
By definition of the $L^2$-gradient $\nabla\mathcal{L}_\partial[u]\in L^2(\Omega)$ and the linearity of the differential operator, the following holds
\begin{equation}
\begin{aligned}
        \langle \nabla\mathcal{L}_\partial[u_2-u_1],u_2-u_1\rangle_{L^2(\Omega)} &=d\mathcal{L}_\partial[u_2-u_1](u_2-u_1)\\&  
        =\langle \Delta (u_2-u_1), j^{k^*} (\Delta (u_2-u_1))\rangle_{L^2(\Omega)}\\&+\langle \Delta (u_2-u_1), j^{k^*} \Delta (u_2-u_1) \rangle_{L^2(\Omega)}.
    \end{aligned}
\end{equation}
Hence 
\begin{equation}
    \langle \nabla\mathcal{L}_\partial[u_2-u_1],u_2-u_1\rangle_{L^2(\Omega)} = 2\|\Delta (u_2-u_1)\|_{H^{-1}(\Omega)}^2.
\end{equation}
Similarly, for the boundary term, it holds
\begin{equation}
    \begin{aligned}
        \langle \nabla\mathcal{L}_b[u_2-u_1],u_2-u_1\rangle_{L^2(\partial\Omega)} &
        =2\|u_2|_{\partial\Omega}-u_1|_{\partial\Omega}\|_{H^{1/2}(\partial\Omega)}^2.
    \end{aligned}
\end{equation}
Therefore, applying the estimates in \eqref{eq:est_nab-1} and \ref{eq:est_l2-1}, we have that 
\begin{equation}
\begin{aligned}
         \frac{1}{2}\langle \nabla\mathcal{L}[u_2-u_1],u_2-u_1\rangle_{\mathcal{H}} & = \|\Delta (u_2-u_1)\|_{H^{-1}(\Omega)}^2+\|u_2|_{\partial\Omega}-u_1|_{\partial\Omega}\|_{H^{1/2}(\partial\Omega)}^2\\&
         \gtrsim \|u_2-u_1\|_{H^1(\Omega)}^2,
\end{aligned}
\end{equation}
with $\langle\cdot,\cdot \rangle_{\mathcal{H}}:=\langle\cdot,\cdot \rangle_{L^2(\Omega)}+\langle\cdot,\cdot \rangle_{L^2(\partial\Omega)}$ the graph inner product, 
yielding the claim that $\mathcal{L}$ is $\lambda$-convex for some $\lambda \in \mathbb{R}_+$. 

To finish the proof we apply Tonelli's direct method introduced in Chapter 1 of~\cite{Maso1993}, to a convex, coercive and lower-semi-continuous functional, yielding the existence of the solution. The uniqueness holds since $\mathcal{L}$ is $\lambda$-convex and therefore strictly convex.  
\end{proof}
Next, we present the proof of Lemma~\ref{lemm:grad_SC}.
\begin{proof}[Proof of Lemma~\ref{lemm:grad_SC}]
Let $\mathbb{D}_{\Delta}\in \mathbb{R}^{|A_{n,d}|\times |A_{n,d}|}$ be the polynomial differentiation matrix associated with the operator $\Delta:\Pi_{n}(\Omega)\rightarrow\Pi_{n}(\Omega)$. Following Proposition~\ref{prop:comp_cub}, we know that $\partial^\beta L_{\alpha}\in \Pi_{n}(\Omega)$ is computable in polynomial-time and therefore 
\begin{equation}
    (\mathbb{D}^\beta)_{i,j}:= \partial^\beta L_{i}(p_j),
\end{equation}
is computable in polynomial-time for all $i,j\in \{1,\dots, |A_{n,d}|\}$.

 We compute the $L^2$-gradient of the loss $\mathcal{L}^n_\theta[\theta]:=\mathcal{L}^n[\hat{u}_\theta]$, by evaluating its first variation in the direction of a perturbation $\delta\theta \in \Theta$, yielding
\begin{equation}\label{eq:var_grad}
    d\mathcal{L}^n_\theta[\theta](\delta \theta) = 2\int\limits_{\Omega}\Delta^*(\Delta \hat{u}_\theta-f_n)\hat{u}_{\delta\theta} dx + 2n\int\limits_{\partial\Omega}( \hat{u}_\theta|_{\partial\Omega}-g_n)\hat{u}_{\delta\theta}|_{\partial\Omega} ds,
\end{equation}
where $\Delta^*:\Pi_{n}(\Omega)\rightarrow\Pi_{n}(\Omega)$ is the adjoint of the Laplacian, restricted to polynomial spaces. 
Consequently, by applying the Sobolev cubatures to the variation in \eqref{eq:var_grad}, the $L^2$-gradient is equal to 
\begin{equation}\label{eq:grad_SC}
 \nabla\mathcal{L}^n_\theta[\theta] = 2\mathbb{T}_\Omega^T\mathbb{D}^T_{\Delta }\mathbb{W}_\Omega(\mathbb{D}_{\Delta }\hat{\mathfrak{u}}_\theta -\mathfrak{f})+2\mathbb{T}^T_{\partial\Omega}\mathbb{W}_{\partial\Omega}(\hat{\mathfrak{u}}_\theta|_{\partial\Omega}-\mathfrak{g}),
\end{equation}
where $\mathfrak{f} = f(P_{n}(\Omega))$, $\mathfrak{g} = g(P_{n}(\partial\Omega))$, $\mathbb{W}_\Omega:=\mathrm{diag}(\{\omega_\alpha\}_{\alpha\in A_{n,d}})$, $\mathbb{W}_\Omega:=\mathrm{diag}(\{\omega_\alpha\}_{\alpha\in A_{n,d}})$, $\mathbb{W}_{\partial\Omega}:=\mathrm{diag}(\{\omega_\alpha\}_{\alpha\in A_{n,d-1}})$, $\mathbb{T}_\Omega:=(T_\alpha(P_{n}(\Omega)))_{\alpha\in A_{n,d}}$, $\mathbb{T}_{\partial\Omega}:=(T_\alpha(P_{n}(\partial\Omega)))_{\alpha\in A_{n}}$, $\hat{\mathfrak{u}}_\theta:=\mathbb{T}_\Omega\theta$, $\hat{\mathfrak{u}}_\theta:=\mathbb{T}_{\partial\Omega}\theta$.
\\
The rest of the proof is a consequence of Proposition~\ref{prop:comp_cub}, which implies that the matrices and vectors from \eqref{eq:grad_SC} are computable in polynomial-time. Hence, the gradient is computable in polynomial-time. 
\end{proof}
Now, we show the non-differentiability of the signed distance function.
\begin{proof}[Proof of Lemma~\ref{lemm:disc_sol}.]
Let $c \in M_S$, and let $x_l, x_r \in S$ be two distinct closest points, i.e.
\[
\|c - x_l\|_1 = \|c - x_r\|_1 = \phi_S(c).
\]

Define
\[
z_l^\delta := c - \delta (c - x_l), 
\qquad 
z_r^\delta := c - \delta (c - x_r),
\]
for $0 < \delta < 1$, and let $x_l^\delta, x_r^\delta \in S$ be the closest points to $z_l^\delta, z_r^\delta$ respectively. Then we have that $z_l^\delta\to c$ and $z_r^\delta\to c$ as $\delta \to 0$. Moreover, since $S$ is compact, there exists a subsequence such that $x_l^\delta\to \overline{x}_l$ and $x_r^\delta \to\overline{x}_r$, for $\overline{x}_l,\overline{x}_r\in S$.

We prove that $\overline{x}_l=x_l$. By optimality of $x_l^\delta$ we have
\begin{align}
    \|z_l^\delta-x_l^\delta\|_1 \leq \|z_l^\delta-x_l\|_1.
\end{align}
Taking the limit as $\delta\to 0$, and using the continuity of the norm, we obtain
\begin{align}
    \|c-\overline{x}_l\|_1 \leq \|c-x_l\|_1.
\end{align}
Since $x_l$ is a closest point to $c$, we also have
\begin{align}
    \|c-x_l\|_1 \leq \|c-\overline{x}_l\|_1,
\end{align}
and therefore $\|c-\overline{x}_l\|_1 = \|c-x_l\|_1$.

The minimizer of $\|c-x\|_1$ is not unique, therefore it can happen that $\overline{x}_l\neq x_l$. To show that this is not the case, we prove that 
\begin{equation}
    \|z_l^\delta -x_l\|_2<\|z_l^\delta-x\|_2,\ \forall x\in S\textrm{, s.t. }\|c-x\|_1 = \|c-x_l\|_1,\ x\neq x_l.
\end{equation}
To show this, let $x\in S$ such that $\|c-x_l\|_1=\|c-x\|_1$, then we compute 
\begin{align}
    \|z_l^\delta -x_l\|_2^2 = (1-\delta)^2\|c-x_l\|_2^2,
\end{align}
and
\begin{align}
    \|z_l^\delta-x\|_2^2= \|c-x\|_2^2-2\delta \langle c-x,c-x_l\rangle_2+\delta^2\|c-x_l\|_2^2.
\end{align}
Since $\|c-x_l\|_2^2 = \|c-x\|_2^2$, we obtain
\begin{align*}
    \|z_l^\delta -x_l\|_2<\|z_l^\delta-x\|_2
    \iff \|c-x_l\|_2^2 > \langle c-x,c-x_l\rangle_2.
\end{align*}
By the Cauchy--Schwarz inequality, $\langle c-x,c-x_l\rangle_2 \leq \|c-x\|_2\|c-x_l\|_2 = \|c-x_l\|_2^2 = ||c-x||_2^2$, and since $x\neq x_l$, the inequality is strict. 

Thus, $x_l$ is the unique closest point to $z_l^\delta$ for $\delta$ sufficiently small, and therefore $\overline{x}_l=x_l$. The proof for $x_r^\delta$ is analogous.

Finally, since $\phi_S$ is Lipschitz, it is differentiable almost everywhere. Hence, we can choose sequences $(z_l^{\delta_k} )_{k\in \mathbb{N}},(z_r^{\delta_k} )_{k\in \mathbb{N}}$ such that $\phi_S$ is differentiable at these points. At such points, the gradient is given by
\[
\nabla \phi_S(z_l^\delta) = \frac{z_l^\delta - x_l}{\|z_l^\delta - x_l\|_2}, 
\qquad
\nabla \phi_S(z_r^\delta) = \frac{z_r^\delta - x_r}{\|z_r^\delta - x_r\|_2}.
\]

Passing to the limit as $\delta \to 0$, we obtain
\[
\nabla \phi_S(z_l^\delta) \to \frac{c - x_l}{\|c - x_l\|_2}, 
\qquad
\nabla \phi_S(z_r^\delta) \to \frac{c - x_r}{\|c - x_r\|_2}.
\]

Since $x_l \neq x_r$, these limits differ, and therefore $\phi_S$ is not differentiable at $c$.
\end{proof}
We finish this section by proving Theorem~\ref{thm:prop_eik_loss}
\vspace{-1em}\begin{proof}[Proof of Theorem~\ref{thm:prop_eik_loss}]\label{proof-thm_eik}
 We start by showing the H\"older continuity of the loss. Using the reverse triangle inequality, the following holds
\begin{align*}
    |\mathcal{L}[u]-\mathcal{L}[v]|&\leq \bigr\|\|\nabla u-\nabla v\|_1\bigr\|_{L^2(\Omega)} + \|u|_S-v|_S\|_{H^{1/2}(S)}\\&
    \leq \|u-v\|_{H^1(\Omega)},
\end{align*}
where we used the triangle inequality and the trace theorem for the last inequality. 
To prove the coercivity estimate, we define $u_S = \mathcal{E}[u|_{S}]\in H^1(\Omega)$ and $u_0:=(u-u_S)\in H^1_0(\Omega)$. Then, by the extension theorem, the triangle and Poincar\'e inequalities, the following holds
\begin{equation}\label{eq:pf1_ceik}
\|u\|_{H^1(\Omega)}\lesssim \|\nabla u\|_{L^2(\Omega)}+\|u_S\|_{H^1(\Omega)}\lesssim \|\nabla u\|_{L^2(\Omega)}+\|u|_S\|_{H^{1/2}(S)}.
\end{equation}
Moreover, by definition of the PDE term and the reverse triangle inequality, we have that 
\begin{align}\label{eq:pf2_ceik}
    \mathcal{L}_\partial[u] &:=  \bigr\|\|\nabla u\|_1-1\bigr\|_{L^2(\Omega)} \geq \left(\|\nabla u\|_{L^2(\Omega)}^2+\sum\limits_{\stackrel{i,j=1}{i\neq j}}^d\int\limits_\Omega |\partial_{x_i}u| |\partial_{x_j}u|dx\right)^{1/2}-|\Omega|\\&
    \geq \|\nabla u\|_{L^2(\Omega)}-|\Omega|.
\end{align}
Hence, combining~\eqref{eq:pf1_ceik} and ~\eqref{eq:pf1_ceik}, the following holds
\begin{align}
     \mathcal{L}[u]+1\gtrsim \mathcal{L}[u]+|\Omega| &\geq \|\nabla u \|_{L^2(\Omega)}+\|u|_S\|_{H^{1/2}(S)}\\&
    \gtrsim \|u\|_{H^1(\Omega)}.
\end{align}
 Note that, in general, the Eikonal loss 
\[
\mathcal{L}:H^{1}(\Omega)\rightarrow\mathbb{R}_{+}
\]
is neither convex nor weakly lower semi-continuous. Consequently, with the variational tools presented in this paper, we cannot guarantee the existence of a minimizer. 

To address this limitation, we can use the classical result that the signed distance function $\phi$ is the viscosity solution of the Eikonal equation \cite{Katzourakis2014-fe}, and consequently the Eikonal loss is minimized at $\phi\in H^1(\Omega)$. 

Alternatively, we can consider a convex relaxation of the loss
\begin{equation}\label{eq:convex_eikonal_loss}
    \mathcal{L}_c[u]:=||\left(||\nabla u||_1-1\right)_+||_{L^2(\Omega)}+||u|_S||_{H^{1/2}(S)},
\end{equation}
with $(u)_+(x):=\max\{u(x),0\}$. This yields a convex, weakly lower-semi continuous functional satisfying the apriori estimate in~\eqref{eq:apr_est_eik}. Consequently by Tonelli's direct method there exists a solution $u^*\in H^1(\Omega)$ to \eqref{eq:convex_eikonal_loss}. 

Finally, as another alternative, we can consider the modified regularized Eikonal equation 
\begin{equation}\label{eq:visc_eik}
\mathcal{L}_\mu[u]:=\left\||\nabla u|^2-1 \right\|_{L^2(\Omega)}^2 + ||u|_S||_{H^{1/2}(S)}^2+\mu ||u||_{H^1(\Omega)}^2.
\end{equation}
In this setting, the additional regularization term $\mu\| u\|_{H^1(\Omega)}^2$
induces local strong convexity of the functional in a neighborhood of its minimizer,
provided that $\mathcal{L}_\mu[u]$ remains sufficiently small.
As a consequence, one may establish the existence of a local minimizer and obtain a
local quadratic growth condition, which in turn can be used to establish $H^1$-computability results.
A detailed analysis of these relaxed variational problems, as well as the convergence
properties of their associated gradient flows with respect to the signed distance function $\phi$,
is beyond the scope of the present work.
\end{proof}
\bibliography{REF.bib}
\bibliographystyle{unsrt}
\end{document}